\documentclass[11pt,utf8,notitlepage]{article}

\usepackage{amsmath}
\usepackage{mathtools}
\usepackage{amsthm} 
\usepackage{amssymb}
\usepackage{mathrsfs}
\usepackage{stmaryrd} 
\usepackage{commath}
\usepackage{bbm} 
\usepackage[perpage]{footmisc} 
\usepackage{tikz-cd}
\usepackage{float}
\usepackage{enumitem}

\DeclareTextFontCommand{\myemph}{\bfseries\em}

\newcommand{\imp}[2]{($#1$) $\Rightarrow$ ($#2$)} 
\DeclareMathOperator*{\colim}{colim} 
\newcommand{\calg}{\mathbf{CAlg}} 
\renewcommand{\hom}{\mathrm{Hom}}  
\newcommand{\opp}{\mathrm{op}} 
\newcommand{\cnerve}{\mathrm{N}_{\bullet}} 
\newcommand{\fun}{\mathbf{Fun}} 
\newcommand{\setc}{\mathbf{\boldsymbol{\mathcal{S}}et}} 
\newcommand{\bigcatinf}{\mathbf{\boldsymbol{\mathcal{C}}AT}_{\infty}} 
\newcommand{\SP}{\mathbf{\boldsymbol{\mathcal{S}}p}} 
\newcommand{\pr}{\boldsymbol{\mathcal{P}}r} 
\newcommand{\prl}{\mathbf{\pr}^{\mathrm{L}}} 
\newcommand{\prr}{\mathbf{\pr}^{\mathrm{R}}} 
\newcommand{\prlc}{\prl_{\omega}} 
\newcommand{\prrc}{\prr_{\omega}} 
\newcommand{\prst}{\mathbf{\pr}^{\mathrm{st}}} 
\newcommand{\prstc}{\prst_{\omega}} 
\newcommand{\opr}[1]{ #1^{\otimes}} 
\newcommand{\module}[1]{\mathbf{Mod}_{#1}} 
\DeclareMathOperator{\frc}{Frac} 
\DeclareMathOperator{\aut}{Aut} 
\newcommand{\Nm}{\mathrm{Nm}} 
\newcommand{\kch}{\boldsymbol{\mathcal{K}}} 
\newcommand{\bkch}{\kch^{\mathrm{b}}} 
\newcommand{\dcat}{\mathcal{D}}
\newcommand{\bdcat}{\dcat^{\mathrm{b}}}
\DeclareMathOperator{\Char}{char} 
\DeclareMathOperator{\gal}{Gal} 
\newcommand{\Z}{\mathbb{Z}} 
\newcommand{\Q}{\mathbb{Q}} 
\DeclareMathOperator{\spec}{Spec} 
\DeclareMathOperator{\spf}{Spf} 
\newcommand{\ets}{_{\text{\'{e}t}}} 
\newcommand{\rgama}[1]{\mathrm{R}\boldsymbol{\Gamma}_{#1}} 
\DeclareMathOperator{\htcat}{h} 
\newcommand{\subss}[1]{ #1_{\bullet}} 
\newcommand{\sset}{\mathbf{s\boldsymbol{\mathcal{S}}et}} 
\newcommand{\map}{\mathrm{Map}}  
\DeclareMathOperator{\spa}{Spa} 
\newcommand{\agmot}{\mathbf{DA}}
\newcommand{\fagmot}{\mathbf{FDA}}
\newcommand{\rigmot}{\mathbf{RigDA}}

\newtheorem{thm}{Theorem}[section]
\newtheorem{pro}[thm]{Proposition}
\newtheorem{cor}[thm]{Corollary}
\newtheorem{lem}[thm]{Lemma}

\theoremstyle{definition}
\newtheorem{df}[thm]{Definition}
\newtheorem{eg}[thm]{Example}

\newtheorem{rmk}[thm]{Remark}
\newtheorem{notation}[thm]{Notation}
\newtheorem{cons}[thm]{Construction}

\makeatletter
\newcommand{\proofpart}[2]{%
  \par
  \addvspace{\medskipamount}%
  \noindent\underline{\bf Step #1: #2}\par\nobreak
  \addvspace{\smallskipamount}%
  \@afterheading
}
\makeatother

\numberwithin{equation}{section}

\usepackage[
backend=biber,
style=alphabetic,
sorting=nyt,
hyperref,
date=year,
backref,
backrefstyle=two,
maxbibnames=99,
maxcitenames=99,
maxalphanames=99,
doi=true,
isbn=false,
url=false,
useprefix=true 
]{biblatex}
\DeclareBibliographyAlias{misc}{online}

\DeclareFieldFormat{title}{\mkbibemph{#1}} 

 \renewbibmacro{in:}{} 

 \renewbibmacro*{journal+issuetitle}{%
   \usebibmacro{journal}%
  \setunit*{\addspace}%
  {\printtext{Vol. {\textbf{\printfield{volume}}}}}
  \setunit{\addspace}%
  \usebibmacro{issue+date}%
  \setunit{\addcomma\space}%
    \iffieldundef{number}
    {}
    {\printtext{No. \printfield{number}}}%
    \setunit{\addcomma\space}%
  \newunit}

\AtEveryBibitem{
  \ifentrytype{online}{
    \clearfield{doi}
    \clearfield{url}
  }{}
}
\usepackage{xcolor}
\definecolor{terblue}{HTML}{5371C6}
\definecolor{haoblue}{HTML}{83CCEB}
\definecolor{terpink}{HTML}{FF9497}

\usepackage[colorlinks=true,linkcolor=terblue,citecolor=haoblue,urlcolor=terpink]{hyperref}

\usepackage{geometry}
\usepackage{indentfirst} 
\usepackage{afterpage}

\usepackage{titlesec}
\usepackage[nottoc]{tocbibind} 
\usepackage[titles]{tocloft} 
\usepackage[T1]{fontenc}
\usepackage[utf8]{inputenc}
\usepackage{tgschola}
\usepackage[mathcal]{euscript}

\usepackage{fancyhdr}
\titleformat{\section}[block]
{\normalfont\Large\bfseries\filcenter}
{\thesection}
{8pt}{}

\titleformat{\subsubsection}[block]
{\bfseries\large\normalsize}
{}
{0pt}{}

\title{\LARGE \textbf{A Weight Structure on Rigid Analytic Motives\\ over a Field}} 
\author{Kaixing Cao}
\date{\today}

\begin{document}

\maketitle
\begin{abstract}
  In this paper, we construct a monoidal weight structure on the stable $\infty$-category of rigid analytic motives over a local field $K$ via Galois descent. This extends the weight structure on the full subcategory of rigid analytic motives with good reduction, which is defined by Binda-Gallauer-Vezzani. As an application, we show that the Hyodo-Kato realization factors through the weight complex functor studied by Bondarko and Sosnilo. In particular, the weight complex yields a spectral sequence converging to the Hyodo-Kato cohomology of smooth quasi-compact $K$-rigid analytic spaces, thereby inducing a weight filtration on it.
\end{abstract}

\tableofcontents\thispagestyle{empty}


\section{Introduction}
\label{sec:introduction}
Bondarko introduced weight structures on triangulated categories in \cite{Bon10a} as a concept dual to that of $t$-structures (see Remark \ref{rmk:compare-to-t-structure}). Subsequently, Sosnilo developed the theory in the context of stable $\infty$-categories in \cite{Sos19, Sos22}.

One of the main motivations for introducing weight structures is precisely their application to the study of motives: in \cite{Bon10a,Heb10,Bon14}, a weight structure, called the Chow weight structure, was constructed on the stable $\infty$-category $\agmot(k)$ of motives (with rational coefficients) over a field $k$. Since the construction of motivic $t$-structures is very difficult and remains largely open, weight structures provide a more accessible framework for organizing and analyzing motives.

On the other hand, Ayoub developed a theory of rigid analytic motives in \cite{ayorigmot,AGV22}, providing a natural analogue of motives in rigid analytic geometry. This naturally raises the question of whether weight structures can also be constructed in this rigid analytic setting. As observed in \cite{BGV25} and \cite[Appendix A]{BKV25}, starting from the Chow weight structure on the algebraic motives, one can construct a good weight structure on a full subcategory $\rigmot_{\mathrm{gr}}(K)$ of the $\infty$-category of rigid analytic motives, i.e., consisting of those that admit good models; see Definition \ref{df:rigmotgr}. Even though the weight structure is defined only on this subcategory, it allows us to give the ``motivic Hyodo-Kato isomorphism'' (\cite[Theorem 4.53]{BGV25}) and get a weight filtration on the Hyodo-Kato cohomologies as a formal result. These indicate that weight structures on rigid analytic motives provide a powerful tool for connecting motivic ideas with $p$-adic cohomology theories.

In this paper, we will extend the weight structure on the full subcategory $\rigmot_{\mathrm{gr}}(K)$ to the whole category. Furthermore, this extended weight structure remains compatible with the symmetric monoidal structure on $\rigmot(K)$ in the sense of \cite{Aokwt}.

\begin{thm}[{Theorem \ref{thm:wt-str-rigmotcat}}]
  \label{thm:intro-wtstr}
  Let $K$ be a complete non-archimedean field with perfect residue field $k$. There is a bounded weight structure $w$ on the stable $\infty$-category $\rigmot(K)_{\omega}$ of compact rigid analytic motives over $K$, and this bounded weight structure is compatible with the monoidal structure and extends the one on $\rigmot_{\mathrm{gr}}(K)_{\omega}$ mentioned above.

  Moreover, it can be extended uniquely to its Ind-completion $\rigmot(K)$ such that the full subcategory $\rigmot(K)_{w \ge 0}$ is closed under small colimits.
\end{thm}

To construct the weight structure in Theorem \ref{thm:intro-wtstr}, we apply Galois descent to relate compact analytic motives over $K$ with analytic motives with good models: every compact rigid motive over $K$ has potentially good reduction (\cite[Theorem
2.5.34]{ayorigmot}, \cite[Proposition 3.7.17]{AGV22}). In particular, the full subcategory of compact analytic motives can be described
as a ``union'' of rigid analytic motives with good reduction up to finite Galois extensions. We rephrase this phenomenon with the following explicit form:

\begin{pro}[{Proposition \ref{pro:rigmot-cptgen-pgd}}]
  \label{pro:intro-cmot-pgr}
  We have a monoidal equivalence in $\calg(\prlc)$
  \[
\rigmot(K) \simeq \colim \rigmot_{\mathrm{gr}}(L)^{\htcat \gal (L/K)}
\]
where $L$ runs through finite Galois extensions of $K$ and terms on the right-hand side are $\gal (L/K)$-equivariant analytic motives (see Definition \ref{df:ht-fix-pt} and Remark \ref{rmk:htfpt-practical-case}).
\end{pro}

Therefore, it suffices to construct a weight structure on $\rigmot_{\mathrm{gr}}(L)^{\htcat \gal (L/K)}$ and glue it to obtain a weight structure on $\rigmot(K)$. The former category is the homotopy fixed points under the Galois action; in other words, it is the limit of $\rigmot_{\mathrm{gr}}(L)$ along the Galois action. Although in general it is very difficult to equip a limit of stable $\infty$-categories with a weight structure, the situation here is special: the translation functors involved are autoequivalences. Thus, we can give an explicit formula for its compact generators, and prove they form a negative class (see Proposition \ref{pro:cpt-gen-gal-inv-gd} and Proposition \ref{pro:proper-sm-under-lad}). In particular, we get a bounded weight structure on its compact part (Proposition \ref{pro:bdwt-from-heart}), which can be extended to the whole category as usual (discussed in \S \ref{subsec:cons-wtstr}).

As an application of Theorem \ref{thm:intro-wtstr}, we can construct a p-adic analytic generalization of
Rapoport-Zink’s weight spectral sequence together with weight filtration in \cite{RZ82}. The key ingredient that allows us to obtain such filtrations from (bounded) weight structures
is the weight complex functor studied in \cite{Bon10a,Sos19}: there is a symmetric monoidal functor
\begin{equation}
  \label{eq:intro-wtcplx}
  \subss{W} \colon \rigmot(K)_{\omega} \to \bkch(\htcat \mathcal{H}_{w})
\end{equation}
where $\bkch(\htcat \mathcal{H}_w)$ is the $\infty$-category of bounded chain complexes in the heart $\htcat \mathcal{H}_w$ of the weight structure constructed in Theorem \ref{thm:intro-wtstr}.

Now we suppose the residue field $k$ of $K$ is a finite field. Under this assumption, we show that the Hyodo-Kato realization functor factors through the weight complex functor \eqref{eq:intro-wtcplx}:

\begin{pro}[Lemma \ref{lem:GHK-wtcplx}]
  \label{pro:intro-HK-wtcplx}
  After forgetting the monodromy operator, the Hyodo-Kato realization functor \eqref{eq:phi-GK-HK}
  \[
\rgama{\mathrm{HK}} \colon \rigmot(K)_{\omega} \to \bdcat_{(\varphi,G_{K})} (\breve{K})
\]
factors through the weight complex functor. Here $\bdcat_{(\varphi,G_{K})} (\breve{K})$ is the derived $\infty$-category of bounded $(\varphi,G_K)$-modules over $\breve{K}$, in the sense of Fontaine.
\end{pro}

Using the standard spectral sequence associated to the weight complex, see \cite[\S 2]{Bon10a} for instance, of a compact rigid motive $M$ over $K$, we get a spectral sequence converging to the Hyodo-Kato cohomology of $M$. In particular, we obtain the weight filtration on it:

\begin{pro}[Weight Filtration, {Proposition \ref{pro:wtfil}, Corollary \ref{cor:wt-fil-rigsp}}]
  \label{pro:intro-wtfil}
  Let $X$ be a smooth quasi-compact rigid analytic space (or more generally, a compact rigid analytic motive) over $K$. Then, for every integer $n\ge 0$, there is a finite increasing filtration $\mathrm{Fil}_{\bullet}^{W}H_{\mathrm{HK}}^n(X) $ on the arithmetic overconvergent Hyodo-Kato cohomology of $X$ satisfying:
  \begin{enumerate}
  \item the filtration is stable under the Frobenius operator and Galois action on $H^n_{\mathrm{HK}}(X)$;
  \item the $k$-th graded piece of the filtration is pure of weight $n+k$, i.e., the Frobenius has eigenvalues with complex
norm $\abs{k}^{(n+k)/2}$;
  \item the monodromy on the cohomology induces a map $\mathrm{gr}_{k}^{W} H_{\mathrm{HK}}^{n}(X) \to \mathrm{gr}_{k-2}^{W} H_{\mathrm{HK}}^{n}(X)$.
  \end{enumerate}
\end{pro}

\begin{rmk}
If $X$ has semistable reduction, then the weight filtration in the previous proposition has already been given in \cite{BGV25}. However, we do not assume any formal model of $X$ here, because we have the extended weight structure in Theorem~\ref{thm:intro-wtstr}.
\end{rmk}

{\paragraph{Notations and Conventions.}%
  We will freely use the language of $\infty$-categories and follow notations in \cite{HTT,HA}. But we don't treat set-theoretic issues carefully\footnote{Those seeking a more precise treatment of set-theoretic matters may refer to \cite{HTT,HA,kerodon}, especially \cite[\href{https://kerodon.net/tag/03PP}{Section 03PP}]{kerodon}.} for convenience. We assume the Grothendieck universe axioms and fix two Grothendieck universes $\mathcal{U}$ and $\mathcal{V}$ with $\mathcal{U} \in \mathcal{V}$. We will refer to a mathematical object as small if it belongs to $\mathcal{U}$; in particular, we talk about presentable $\infty$-categories with respect to $\mathcal{U}$.

  We will denote by $\bigcatinf$ the $\infty$-category of locally small $\infty$-categories. In particular, we will only use the $\infty$-category of small spaces, denoted by $\mathcal{S}$ (a.k.a., $\infty$-category of anima or $\infty$-groupoids), and its stabilization, called the $\infty$-category of spectra, will be denoted by $\SP$. We denote by $\prl$ (resp. $\prr$) the sub-$\infty$-category of $\bigcatinf$ spanned by presentable $\infty$-categories and left adjoints (resp. right adjoints). Furthermore, we let $\prlc$ (resp. $\prrc$) denote the sub-$\infty$-categories of $\prl$ (resp. $\prr$) spanned by compactly generated $\infty$-categories and compact-preserving functors (resp. functors commuting with filtered colimits). The full subcategory of $\prl$ (resp. $\prlc$) spanned by presentable (resp. compactly generated) stable $\infty$-categories will be denoted by $\prst$ (resp. $\prstc$).

  Let $\mathcal{C} \in \bigcatinf$. Its homotopy category will be denoted by $\htcat \mathcal{C}$. Let $X, Y$ be objects in $\mathcal{C}$. We will denote by $\map_{\mathcal{C}}(X,Y)$ (resp. $\mathrm{map_{\mathcal{C}}(X,Y)}$), or simply $\map(X,Y)$ (resp. $\mathrm{map}(X,Y)$) if $\mathcal{C}$ is clear from the context, the mapping space (resp. the mapping spectrum) from $X$ to $Y$ in $\mathcal{C}$. We will also write $\pi_0 \map_{\mathcal{C}}(X,Y) \cong \hom_{\htcat \mathcal{C}}(X,Y)$ simply as $\hom_{\mathcal{C}}(X,Y)$. We denote by $\mathcal{C}_{\omega}$ the full subcategory of $\mathcal{C}$ spanned by compact objects.

  All schemes, formal schemes and rigid analytic spaces are assumed to be ($\mathcal{U}$)-small. For motives, we only use \'{e}tale motives with rational coefficients; for homological algebras, our conventions are homological throughout.
}

{\paragraph{Acknowledgments.}%
  The author would like to sincerely thank his advisor, Alberto Vezzani, for suggesting the topic of this paper and for many helpful conversations throughout the development of the work. He is also grateful to Veronika Ertl for valuable suggestions on the exposition and for comments that improved the clarity of the writing.
}
\section{\texorpdfstring{Preliminaries on $\infty$-Category
    Theory}{Preliminaries on Infinite-Category Theory}}
In this section, we recall some results of $\infty$-categories. In \S \ref{subsec:htfixpt}, we review the theory of group actions and their fixed points in the setting of $\infty$-categories by following \cite[\S 6.1.6]{HA} and \cite[Chapter I]{NSTCH18}. This part will be useful for us to construct weight structures on rigid analytic motives.

\subsection{\texorpdfstring{Limits and Colimits of $\infty$-Categories}{Limits and Colimits of Infinite-Categories}}
\label{subsec:colim-infcat}
We begin by recalling how to compute the limits/colimits of presentable $\infty$-categories.

\begin{pro}[{\cite[Proposition 5.5.3.13, Theorem 5.5.3.18]{HTT}}]
  \label{pro:lim-prl-prr}
  The $\infty$-categories $\prl$ and $\prl$ admit small
  limits. Moreover, the inclusion functors
  \[
\prl \hookrightarrow \bigcatinf, \quad \prr \hookrightarrow \bigcatinf
\]
preserve small limits.
\end{pro}

\begin{pro}[{\cite[Proposition 5.5.7.6 and Its Proof]{HTT}}]
  \label{pro:lim-in-prrc}
  The $\infty$-category $\prrc$ admits small limits and the inclusions $\prrc \hookrightarrow \prr \hookrightarrow \bigcatinf$ preserve small limits. Moreover, let
  $p\colon K \to \prrc$ be a diagram of compactly generated
  $\infty$-categories $\{\mathcal{C}_{\alpha}\}$ with a limit
  $\mathcal{C}$ in $\prr$, then $\mathcal{C}$ is also compactly
  generated by images of $\mathcal{C}_{\alpha,\omega}$ under $F_{\alpha}$, where
  $F_{\alpha}\colon \mathcal{C}_{\alpha} \to \mathcal{C}$ is the left
  adjoint of the canonical functor $G_{\alpha}\colon \mathcal{C} \to \mathcal{C}_{\alpha}$.
\end{pro}

We recall an explicit formula to understand (small) limits of $\infty$-categories in a computable way.

\begin{pro}[Mapping Spaces of Limits]
  \label{pro:mapsp-of-lim}
  Let $\mathcal{F}\colon I \to \bigcatinf$ be a diagram of $\infty$-categories indexed by a simplicial set $I$. For each vertex $i$ of $I$, we let $\mathcal{C}_i$ denote the image of $i$ under $\mathcal{F}$ and put $\mathcal{C}= \lim_{\mathcal{F}} \mathcal{C}_i$, whose canonical functors are denoted by $p_i\colon \mathcal{C} \to \mathcal{C}_i$. Given any objects $X$, $Y$ in $\mathcal{C}$ with $X_i=p_i(X)$, $Y_i=p_i(Y)$, then we have a homotopy equivalence of spaces
  \[
\map_{\mathcal{C}}(X,Y) \simeq \lim \map_{\mathcal{C}_i}(X_i,Y_i).
\]
\end{pro}
\begin{proof}
  For each $i \in I$, we have a pullback diagram
  \begin{figure}
[H]\centering
\begin{tikzcd}
{\map_{\mathcal{C}_i}(X_i,Y_i)} \arrow[r] \arrow[d] & {\mathrm{Fun}(\Delta^1, \mathcal{C}_i)} \arrow[d] \\
\Delta^0 \arrow[r, "{(X_i,Y_i)}"]                   & {\mathrm{Fun}(\partial\Delta^1, \mathcal{C}_i)}  
\end{tikzcd}
\end{figure}\noindent
in the (ordinary) category of simplicial sets $\sset$. The exponentiation of isofibrations (\cite[\href{https://kerodon.net/tag/01F3}{Corollary 01F3}]{kerodon}
) shows the right vertical map is an isofibration. This implies that this diagram is a categorical pullback square, i.e., $\map_{\mathcal{C}_i}(X_i,Y_i)$ is equivalent to the homotopy fiber product (see \cite[\href{https://kerodon.net/tag/033P}{Proposition 033P}]{kerodon}). After taking the homotopy coherent nerve, we get a pullback square in $\bigcatinf$, where the commutative diagram above represents the boundary of this pullback square. The limit of these squares yields a pullback square in $\bigcatinf$ whose boundary is given by
\begin{figure}[H]
  \centering
\begin{tikzcd}
{\lim\map_{\mathcal{C}_i}(X_i,Y_i)} \arrow[r] \arrow[d] & {\fun(\Delta^1,\mathcal{C})} \arrow[d] \\
\Delta^0 \arrow[r, "{(X,Y)}"]                           & {\fun(\partial\Delta^1,\mathcal{C})}  
\end{tikzcd}.
\end{figure}\noindent
This follows from the facts that limits of functors are computed levelwise and that limits commute with limits. This gives the desired homotopy equivalence of spaces.
\end{proof}

In the sequel, we will meet many symmetric monoidal $\infty$-categories (like $\infty$-categories of motives). We also collect some results about small limits/colimits of such categories. For this purpose, we consider the $\infty$-categories of commutative algebra objects in $\prl$ (resp. $\prlc$), denoted by $\calg(\prl)$ (resp. by $\calg(\prlc)$). More precisely, an object in $\calg(\prl)$ (resp. $\calg(\prlc)$) is a presentable (resp. compactly generated) symmetric monoidal $\infty$-category whose tensor product functor $-\otimes-$ commutes with small colimits separately in each variable; see \cite[Remark 4.8.1.24 \& Lemma 5.3.2.11]{HA}

\begin{pro}
  \label{pro:clim-calgpr}
  We let $\mathcal{C}$ be either $\prl$ or $\prlc$ with the forgetful functor $\theta\colon \calg(\mathcal{C}) \to \mathcal{C}$.
  \begin{enumerate}
  \item The forgetful functor $\theta\colon \calg(\mathcal{C}) \to \mathcal{C}$ is conservative.
    
  \item The $\infty$-category $\calg(\mathcal{C})$ admits sifted colimits. Moreover, the forgetful functor $\theta$ preserves and reflects these colimits.
    
  \item The $\infty$-category $\calg(\mathcal{C})$ admits small limits. Moreover, the forgetful functor $\theta$ preserves and reflects small limits.
\end{enumerate}
\end{pro}
\begin{proof}
  The conservativity follows from \cite[Lemma 3.2.2.6]{HA}. As we saw in Proposition \ref{pro:lim-prl-prr} and Proposition \ref{pro:lim-in-prrc}, $\prl$ admits small colimits and limits and $\prlc$ admits small colimits. In fact, $\prlc$ also admits small limits using the equivalence in \cite[Proposition 5.5.7.8]{HTT}. Therefore, part $(2)$ and part $(3)$ follow from \cite[Corollary 3.2.2.5, Corollary 3.2.3.2]{HA}.
\end{proof}

\begin{rmk}
  \label{rmk:lim-prst}
  We may consider the above results in the setting of presentable stable $\infty$-categories. For this, we need to guarantee the existence of small limits in $\calg(\prst)$ (resp. $\calg(\prstc)$): these limits reduce to limits in $\prst$ (resp. $\prstc$), while the natural inclusion functor $\prst \hookrightarrow \prl$ (resp. $\prstc \hookrightarrow \prlc$) preserves limits by \cite[Theorem 1.1.4.4]{HA}.
\end{rmk}

\subsection{Homotopy Fixed Points}
\label{subsec:htfixpt}

\begin{notation}
  Given a (multiplicative) group $G$, we let $BG$ denote the groupoid defined by $G$. To be precise, $BG$ has exactly one object $*$ with $\hom(*,*)=G$, and the composition law is given by the multiplication of $G$. The $\infty$-groupoid $\subss{B} G$ is obtained by taking the nerve of it.
\end{notation}

\begin{df}
\label{df:ht-fix-pt}
  Let $G$ be a group and $\mathcal{C}$ an $\infty$-category with an object $X$.
  \begin{enumerate}
  \item A \myemph{$G$-action}\index{group action on $\infty$-categories} on $X$ is a functor $B_{\bullet}G \to \mathcal{C}$ sending $*$ to $X$. We let $\mathcal{C}^{BG}:= \mathrm{Fun}(B_{\bullet} G, \mathcal{C})$ denote the \myemph{$\infty$-category of $G$-objects} in $\mathcal{C}$.
    \index{$G$-objects}
    \index[sym]{$\mathcal{C}^{BG}$ ($\infty$-category of $G$-objects in $\mathcal{C}$)}

  \item Assume $\mathcal{C}$ admits all limits indexed by $\subss{B} G$. The \myemph{homotopy fixed point} functor is given by
    \begin{align*}
      (-)^{\htcat G} \colon \mathcal{C}^{BG}&\to \mathcal{C}\\
      (\rho\colon \subss{B} G \to \mathcal{C})&\mapsto \lim \rho. 
    \end{align*}
    \index{homotopy fixed point}
    \index[sym]{$X^{\htcat G} (homotopy fixed point)$}
\end{enumerate}
\end{df}

\begin{rmk}
  \label{rmk:htfpt-practical-case}
  \begin{enumerate}
  \item In the following, we take $\mathcal{C}= \bigcatinf$, $\prl$ or $\prr$ in the Definition \ref{df:ht-fix-pt}. This, in particular, defines group actions on $\infty$-categories and the corresponding $\infty$-category of homotopy fixed points. In practice, we will consider the case where $X= \rigmot(K)$ for a complete non-archimedean field $K$, which serves as the primary example in this paper.
    
  \item Although $\subss{B} G$ is the nerve of a $1$-category, giving a $G$-action on an object $X \in \mathcal{C}$ is not simply given by a group homomorphism $G \to \aut_{\htcat \mathcal{C}}(X)$ which is only equivalent to a morphism $\mathrm{sk}_2(\subss{B} G) \to \mathcal{C}$. For us, we are interested in a special case $\mathcal{C}= \bigcatinf$; therefore, using Grothendieck's construction, a functor $\subss{B}G \to \bigcatinf$ is given by (up to isomorphism) taking the simplicial nerve of a functor from $BG$ to the simplicial category of $\infty$-categories, see \cite[\href{https://kerodon.net/tag/0387}{Corollary 0387}]{kerodon}. Therefore, for an $\infty$-category $\mathcal{C}$ with a $G$-action, we often write it as a pair $\tilde{\mathcal{C}}=(\mathcal{C}, \rho)$, where $\rho$ represents the $G$-action on $\mathcal{C}$.
\end{enumerate}
\end{rmk}

From now on, we focus on group actions on $\infty$-categories. We will use Proposition \ref{pro:mapsp-of-lim} to understand mapping spaces of their $\infty$-categories of homotopy fixed points. As a preparatory step, we first recall how to compute homotopy groups of homotopy fixed points of $G$-spectra.

\begin{pro}[Homotopy Fixed Points Spectral Sequence]
  \label{pro:sseq-htfixpt}
  Given a finite group $G$, let $E$ be a $G$-spectrum (i.e. a $G$-object in $\SP$). Then there is a convergent spectral sequence
  \[
E_{p,q}^2\simeq H^{-p}(G, \pi_q(E)) \Rightarrow \pi_{p+q}(E^{\htcat G}).
\]
\end{pro}
\begin{proof}
This is a special case of the Bousfield-Kan spectral sequence in \cite{BK72}; more precisely, we apply the Bousfield-Kan spectral sequence to the totalization of the cosimplicial space $\map(\subss{E} G, X)^{\htcat G}$; see also \cite[Proposition 7.7]{GJ09}.
\end{proof}

\begin{cor}
  \label{cor:mapsp-htfixpt}
  Let $\mathcal{C}$ be an $\infty$-category with an action by a finite group $G$. Given a pair of objects $\bar{M}$, $\bar{N}$ in $\mathcal{C}^{\htcat G}$, we let $M$, $N$ denote their underlying objects in $\mathcal{C}$. Then there is a homotopy equivalence of spaces
  \[
\map_{\mathcal{C}^{\htcat G}}(\bar{M}, \bar{N}) \simeq \map_{\mathcal{C}}(M,N)^{\htcat G}.
\]
 Moreover, if $\mathcal{C}$ is $\Q$-linear, then, for each integer $i \ge 0$, we have an isomorphism
\[
\pi_i \map_{\mathcal{C}^{\htcat G}}(\bar{M}, \bar{N}) \simeq \left( \pi_i \map_{\mathcal{C}}(M,N) \right)^{G}.
\]
\end{cor}
\begin{proof}
  The first assertion follows from Proposition \ref{pro:mapsp-of-lim}. Under the homotopy equivalence in the first assertion, we apply Proposition \ref{pro:sseq-htfixpt} to $E=\map_{\mathcal{C}}(M,N)$: there is a convergent spectral sequence 
  \[
E^2_{p,q}=H^{-p} \left( G, \pi_q \map_{\mathcal{C}}(M,N) \right) \Rightarrow \pi_{p+q} \map_{\mathcal{C}^{\htcat G}}(\bar{M}, \bar{N}).
\]
Now assume $\mathcal{C}$ is $\Q$-linear; then each term $E^2_{p,q}$ on the second page is a $\Q$-vector space. The vanishing of higher Galois cohomologies of $\Q$-vector spaces (\cite[\href{https://stacks.math.columbia.edu/tag/0DV3}{Lemma
  0DV3}]{stacks-project}) implies that the spectral sequence collapses on the second page with non-trivial terms $E^2_{0,i}= \left( \pi_i \map_{\mathcal{C}}(M,N) \right)^{G}$. Therefore, the convergence of the spectral sequence gives isomorphisms.
\end{proof}

It is easy to deduce, from Corollary \ref{cor:mapsp-htfixpt}, that taking homotopy fixed points preserves full faithfulness of functors. To clarify it, we introduce the following notion:

\begin{df}
  \label{df:G-subcat}
  \index{$G$-subcategory}
  Let $G$ be a group and let $ \mathrm{ev}\colon \bigcatinf^{BG} \to \bigcatinf$ denote the evaluation functor at the point, i.e., induced by applying the functor $\fun(-, \bigcatinf)$ to the obvious map $\Delta^0 \to \subss{B} G$. Given a $G$-$\infty$-category $\tilde{\mathcal{C}}=(\mathcal{C}, \rho)$ (see Remark~\ref{rmk:htfpt-practical-case} $(2)$), a \myemph{$G$-subcategory} of $\tilde{\mathcal{C}}$ is a $G$-$\infty$-category $\tilde{D}=(\mathcal{D}, \tau)$ together with a functor $\tilde{\iota}\colon \tilde{\mathcal{D}} \to \tilde{\mathcal{C}}$ such that $\iota:=\mathrm{ev}(\tilde{\iota})$ is a fully faithful functor.
\end{df}

\begin{rmk}
  \label{rmk:unique-subGcat}
  In Definition \ref{df:G-subcat}, the $G$-action on $\mathcal{D}$ is unique (up to homotopy) because the functor $\tilde{\iota}\colon \tilde{D} \to \tilde{C}$ is a Cartesian morphism with respect to the evaluation functor $\mathrm{ev}$. Indeed, for every $G$-$\infty$-category $\tilde{E}$, the fully faithful functor $\iota\colon \mathcal{D} \to \mathcal{C}$ induces a fully faithful functor (see \cite[Lemma 5.2]{GHN17} for a stronger form)
  \[
\map(\mathcal{E}, \mathcal{D}) \to \map(\mathcal{E}, \mathcal{C}).
\]
Therefore, the commutative diagram of Kan complexes
\begin{figure}
[H]\centering
\begin{tikzcd}
{\map(\tilde{\mathcal{E}}, \tilde{\mathcal{D}})} \arrow[r] \arrow[d] & {\map(\tilde{\mathcal{E}}, \tilde{\mathcal{C}})} \arrow[d] \\
{\map(\mathcal{E}, \mathcal{D})} \arrow[r, hook]                     & {\map(\mathcal{E}, \mathcal{C})}                          
\end{tikzcd}
\end{figure}\noindent
is a homotopy pullback square.
\end{rmk}

\begin{cor}
  \label{cor:grp-inv-preserve-ff}
Let $\mathcal{C}$ be an $\infty$-category equipped with an action by a
finite group $G$. Assume $\mathcal{D}$ is a $G$-subcategory of $\mathcal{C}$. Then the functor
\[
\mathcal{D}^{\htcat G} \to \mathcal{C}^{\htcat G}
\]
obtained by applying $(-)^{\htcat G}$ to the canonical inclusion, is fully faithful.
\end{cor}
\begin{proof}
 Consider the commutative diagram (up to homotopy)
 \begin{figure}
[H]\centering
\begin{tikzcd}
\mathcal{D}^{\htcat G} \arrow[r] \arrow[d] & \mathcal{C}^{\htcat G} \arrow[d] \\
\mathcal{D} \arrow[r, hook]                & \mathcal{C}                     
\end{tikzcd}
\end{figure}\noindent
where vertical morphisms are forgetful functors. Let $F\colon
\mathcal{D}^{\htcat G} \to \mathcal{C}^{\htcat G}$
denote the induced functor and let $\bar{M}$ and $\bar{N}$ be two objects in
$\mathcal{D}^{\htcat G}$ with underlying objects $M$, $N$ in
$\mathcal{D}$. The commutative diagram above shows the underlying objects of $F(\bar{M})$, $F(\bar{N})$ in
$\mathcal{C}$ are $\iota(M)$, $\iota(N)$, respectively; here $\iota \colon \mathcal{D} \hookrightarrow \mathcal{C}$ is the inclusion. We deduce from Corollary \ref{cor:mapsp-htfixpt}:
\[
\map_{\mathcal{D}^{\htcat G}}(\bar{M}, \bar{N}) \simeq
\map_{\mathcal{D}}(M,N)^{\htcat G} \simeq
\map_{\mathcal{C}}(\iota(M),\iota(N))^{\htcat G} \simeq \map_{\mathcal{C}^{\htcat
    G}}(F(\bar{M}), F(\bar{N})).
\]
\end{proof}

We end the subsection with a useful and well-known fact about $\subss{B} G$.

\begin{lem}
  \label{lem:BG-georeal}
  Given a group $G$, we have a homotopy equivalence in $\mathcal{S}$,
  \[
\subss{B} G \simeq \abs{\subss{G}},
\]
where $\subss{G}$ is the simplicial object in $\mathcal{S}$ given by
\[
\cnerve(\Delta^{\mathrm{op}}) \xrightarrow{\subss{B} G} \cnerve(\setc) \hookrightarrow \mathcal{S}.
\]
\end{lem}
\begin{proof}
It follows from the homotopy equivalence $\subss{B} G \simeq \subss{\mathrm{Sing}}(\,\abs{\subss{B} G})$ and the fact that the latter is just computed by the given geometric realization; for instance, see \cite[\href{https://kerodon.net/tag/04QS}{Variant 04QS}]{kerodon}.
\end{proof}


\section{\texorpdfstring{Weight Structures on Stable $\infty$-Categories}{Weight Structures on Stable Infinite-Categories}}
For the reader's convenience, we review the basic theory of weight structures on stable $\infty$-categories, after which we will recall several constructions of it. The main reference for this section is \cite{Bon10a,BS19,Sos19,Sos22}, to which we refer the reader for further details.

\begin{notation}
  Let $\mathcal{C}$ be a stable $\infty$-category and $\mathcal{D}$ a full subcategory. We define a full subcategory $\prescript{\perp}{}{\mathcal{D}}$ of $\mathcal{C}$ that is spanned by those objects $X \in \mathcal{C}$ such that
  \[
\hom_{\mathcal{C}}(X, Y) \simeq 0 , \text{ for any $Y \in \mathcal{D}$}.
\]
We will call $\prescript{\perp}{}{\mathcal{D}}$ the \myemph{left orthogonal complement} of $\mathcal{D}$ in $\mathcal{C}$.

Similarly, we can define the \myemph{right orthogonal complement} $\mathcal{D}^{\perp}$ of $\mathcal{D}$ in $\mathcal{C}$ as the full subcategory of $\mathcal{C}$ spanned by those objects $X \in \mathcal{C}$ such that
\[
\hom_{\mathcal{C}}(Y, X) \simeq 0 , \text{ for any $Y \in \mathcal{D}$}.
\]
\end{notation}

\subsection{The Definition of Weight Structures}
\begin{df}[{\cite[Definition 1.1.1]{Bon10a}, \cite[Definition 3.1.1]{Sos22}}]
A \myemph{weight structure}\index{weight structure} on a stable
$\infty$-category $\mathcal{C}$ consists of a pair of retract-closed full subcategories $w=(\mathcal{C}_{w \le 0},\mathcal{C}_{w \ge 0})$ satisfying the following properties:
\begin{description}

\item[Semi-invariance] $ \mathcal{C}_{w\ge 0}[1] \subseteq \mathcal{C}_{w\ge
    0}$, $ \mathcal{C}_{w\le 0}[-1] \subseteq \mathcal{C}_{w \le
    0}$; and we write, for any integers $n \in \Z $,
  \[
\mathcal{C}_{w \ge n}:=  \mathcal{C}_{w \ge 0} [n], \qquad
\mathcal{C}_{\le n}:=  \mathcal{C}_{w \le 0}[n].
\]
  
\item[Weak Orthogonality] If $X \in \mathcal{C}_{w\le 0}$ and $Y\in \mathcal{C}_{w
    \ge 1}$, then $\hom_{\mathcal{C}}(X,Y)\simeq 0$.
  
\item[Weight Decomposition] For any object $X \in \mathcal{C}$, we
  have a cofiber sequence
  \[
X_{\le 0} \to X \to X_{\ge 1},
\]
where $X_{\le 0} \in \mathcal{C}_{w\le 0}$ and $X_{\ge 1} \in
\mathcal{C}_{w \ge 1}$. This is called a \myemph{weight decomposition}\index{weight decomposition}
of $X$.
\end{description}
In the case, we also say $\mathcal{C}$ is a \myemph{weighted stable
  $\infty$-category}.
\index{weighted stable $\infty$-category}
\end{df}

\begin{rmk}[Dual to $t$-Structures]
  \label{rmk:compare-to-t-structure}
  In \cite{Pau08}, weight structures were independently studied under the name ``co-t-structure''. In fact, the axioms of a weight structure are dual to those of a t-structure
  on a stable $\infty$-category (or a triangulated category) $\mathcal{C}$: apart from the semi-invariance condition, the role of the pair $(\mathcal{C}_{t \ge 0}, \mathcal{C}_{t\le -1})$ in the axioms of a $t$-structure corresponds to the pair $(\mathcal{C}_{w\le 0}, \mathcal{C}_{w \ge 1})$ for a weight structure.
\end{rmk}

\begin{df}
  Let $\mathcal{C}$ be a stable $\infty$-category equipped with a weight structure $w=(\mathcal{C}_{w\le 0},\mathcal{C}_{w\ge 0})$.
  \begin{enumerate}
  \item For every pair $(a,b)$ of integers with $a \le b$, we let $\mathcal{C}_{[a,b]}$ denote the full subcategory of $\mathcal{C}$ spanned by $\mathcal{C}_{w \ge a} \cap \mathcal{C}_{w \le b}$.
    \index[sym]{$(\mathcal{C}_{[a,b]})$}
    In particular, if $a=b=n$, we write it simply by $\mathcal{C}_{w=n}$.
    \index[sym]{$\mathcal{C}_{w=n}$}
    
  \item We let $\mathcal{C}^{\mathrm{b}}$ denote the full subcategory of $\mathcal{C}$ spanned by those objects $X $ satisfying $X \in \mathcal{C}_{[a,b]}$ for some integers $a \le b$. The weight structure $w$ on $\mathcal{C}$ is said to be
    \myemph{bounded}\index{weight structure!bounded}
    if $\mathcal{C}= \mathcal{C}^{\mathrm{b}}$.
    \index[sym]{$\mathcal{C}^{\mathrm{b}}$}
    
  \item We will write $\mathcal{C}_{w=0}$ by $\mathcal{C}^{\heartsuit_w}$ and call it the \myemph{heart} of the weight structure $w$.\index{weight structure!heart}
    \index[sym]{$\mathcal{C}^{\heartsuit_w}$}
\end{enumerate}
\end{df}

\begin{eg}
  \label{eg:wtstr-chain}
  Let $\mathcal{A}$ be an additive category. The stable $\infty$-category $\bkch(\mathcal{A})$ of bounded chain complexes has a canonical bounded weight structure:
  \begin{align*}
    \bkch(\mathcal{A})_{w \ge 0}&= \left\{ \,  M \in \bkch(\mathcal{A}) \middle|\, M \text{ is isomorphic to a complex whose all negative degrees vanish}  \,\right\}\\
    \bkch(\mathcal{A})_{w \le 0}&= \left\{ \,  M \in \bkch(\mathcal{A}) \middle|\, M \text{ is isomorphic to a complex whose all positive degrees vanish}  \,\right\}
  \end{align*}
  The heart of this weight structure is exactly $\cnerve(\mathcal{A})$.
\end{eg}

Like $t$-structures, we also keep track of how weights vary under functors.

\begin{df}[{\cite[Definition 1.2.1]{Bon14}}]
  \label{df:wtexfun}
Let $(\mathcal{C}, w_{\mathcal{C}})$ and $(\mathcal{D}, w_{\mathcal{D}})$ be weighted stable
$\infty$-categories and let $F\colon \mathcal{C} \to \mathcal{D}$ be
an exact functor of the underlying stable $\infty$-categories. We say $F$ is \myemph{left or right weight-exact} if $F(\mathcal{C}_{w\ge 0}) \subseteq \mathcal{D}_{w\ge 0}$ or $F(\mathcal{C}_{w\le 0}) \subseteq \mathcal{D}_{w\le 0}$, respectively. If $F$ is both left and right weight-exact, then we say $F$ is \myemph{weight-exact}.
\end{df}

\begin{rmk}
  \label{rmk:wtex-for-bd}
  In Definition \ref{df:wtexfun}, if $w_{\mathcal{C}}$ is a bounded weight structure on $\mathcal{C}$, then the weight-exactness of $F$ can be checked using $F(\mathcal{C}^{\heartsuit_w})$. To be precise, if $w_{\mathcal{C}}$ is bounded, then $F$ is weight-exact if and only if the essential image of $\mathcal{C}^{\heartsuit_w}$ under $F$ lies in $\mathcal{D}^{\heartsuit_{w}}$ by \cite[Remark 1.2.3 (10)]{Bon14}; see also Proposition \ref{pro:bdwt-from-heart}.
\end{rmk}


\begin{df}[{\cite[Definition 4.1]{Aokwt}}]
  \index{weight structure!compatible with monoidal structure}
Let $\opr{\mathcal{C}}$ be a stable symmetric monoidal $\infty$-category. We say a weight structure $w=(\mathcal{C}_{w\ge 0}, \mathcal{C}_{w\le 0})$ on the underlying $\infty$-category $\mathcal{C}$ is \myemph{compatible} with the symmetric monoidal structure if two full subcategories $\mathcal{C}_{w \ge 0}$ and $\mathcal{C}_{w \le 0}$ are stable under the tensor product functor and contain the tensor unit.
\end{df}

\begin{rmk}
  \label{rmk:monoidal-wt}
  If a weight structure on the underlying $\infty$-category of a stable symmetric monoidal $\infty$-category $\opr{\mathcal{C}}$ is compatible with the monoidal structure, then, by \cite[Proposition 2.2.1.1]{HA}, the symmetric monoidal structure can be restricted to full subcategories $\mathcal{C}_{w\ge 0}$ and $\mathcal{C}_{w\le 0}$, $\mathcal{C}^{\heartsuit_w}$. We will use the obvious notations $\opr{\mathcal{C}}_{w\ge 0}$, $\opr{\mathcal{C}}_{w \le 0}$, $\opr{(\mathcal{C}^{\heartsuit_w})}$ to denote the monoidal restrictions.
\end{rmk}


\subsection{Bounded Weight Structures and Weight Complex Functors}
\label{sec:bound-weight-struct}

We focus on bounded weight structures in the subsection. We will recall how these weight structures are recovered by their hearts.

\begin{df}
Let $\mathcal{C}$ be a stable $\infty$-category. A full subcategory $\mathcal{N}$ of $\mathcal{C}$ is called \myemph{negative}\index{negative full subcategory} if the mapping space $\map_{\mathcal{C}}(X,Y)$ is connected for all $X, Y \in \mathcal{N}$.
\end{df}

As shown in \cite[Corollary 1.5.7]{Bon10a}, the heart of a bounded weight structure is additive, idempotent complete and negative. In fact, these conditions completely determine the bounded weight structures:

\begin{pro}[{\cite[Theorem 4.3.2 (II)]{Bon10a}}]
  \label{pro:bdwt-from-heart}
  Let $\mathcal{C}$ be a stable $\infty$-category. If $\mathcal{N}$ is a full subcategory of $\mathcal{C}$ satisfying
  \begin{itemize}
  \item the full subcategory $\mathcal{N}$ is negative;
  \item the full subcategory $\mathcal{N}$ generated $\mathcal{C}$ under finite colimits, negative shifts and retracts;
  \end{itemize}
then there is a unique bounded weight structure on $\mathcal{C}$ whose heart
contains $\mathcal{N}$. Moreover, the heart is the idempotent completion of $\mathcal{N}$.

\end{pro}

\begin{cor}
  \label{cor:tensor-on-bdheart}
Let $\opr{\mathcal{C}} \in \calg(\prl)$. Assume its underlying $\infty$-category $\mathcal{C}$ is stable and equipped with a bounded weight structure $w$. Then the weight structure $w$ is compatible with the monoidal structure if and only if $\mathcal{C}^{\heartsuit_{w}}$ contains the tensor unit and is stable under tensor products.
\end{cor}
\begin{proof}
The ``only if'' part is obvious (see Remark \ref{rmk:monoidal-wt}). Conversely, since $\opr{\mathcal{C}} \in \calg(\prl)$, the tensor products preserve small colimits in each variable; in particular, they are exact (see \cite[Proposition 1.1.4.1]{HA}). So we can conclude from Proposition \ref{pro:bdwt-from-heart}.
\end{proof}

\begin{rmk}
  \label{rmk:bdwt-heart}
A stronger result in a categorical formulation can be found in \cite[Proposition 3.3]{Sos19}: Sosnilo proved that taking weight heart is a fully faithful functor from the $\infty$-category of boundedly weighted stable $\infty$-categories to the $\infty$-category of additive $\infty$-categories. The essential image of this embedding is also given in \cite[Corollary 3.4]{Sos19}.
\end{rmk}

There is an important construction related to bounded weight structures:

\begin{cons}[Weight Complex Functor, {\cite[\S 3]{Bon10a}, \cite[Corollary 3.5]{Sos19}}]
  \label{cons:wt-cplx}
  Let $\mathcal{C}$ be a boundedly weighted stable $\infty$-category. Recall that, in Example \ref{eg:wtstr-chain}, we have seen that $\cnerve(\htcat \mathcal{C}^{\heartsuit_{w}})$ is the heart of the bounded weighted stable $\infty$-category $\bkch(\htcat \mathcal{C}^{\heartsuit_w})$. Remark \ref{rmk:bdwt-heart} tells us there is a weight-exact functor
    \[
\subss{W}\colon \mathcal{C} \to \bkch(\htcat \mathcal{C}^{\heartsuit_{w}}),
\]
whose heart is the unit functor $\mathcal{C}^{\heartsuit_w} \to \cnerve(\htcat \mathcal{C}^{\heartsuit_{w}})$. This functor is called the \myemph{weight complex functor} of $\mathcal{C}$. Moreover, if $\mathcal{C}$ admits a symmetric monoidal structure compatible with the weight structure, then the weight complex functor has the symmetric monoidal refinement (\cite[Theorem 4.3]{Aokwt}, \cite[Proposition 3.27]{BGV25}).
\end{cons}

\subsection{Constructions of Weight Structures}
\label{subsec:cons-wtstr}
There is a generalization of Proposition \ref{pro:bdwt-from-heart} allowing us to construct a weight structure from a negative full subcategory. However, these weight structures are not necessarily bounded:

\begin{pro}[{\cite[Theorem 2.2.1]{BS19}}]
  \label{pro:wtstr-from-neg}
  Let $\mathcal{C}$ be a stable $\infty$-category admitting small colimits. Assume $\mathcal{N}$ is a full subcategory of $\mathcal{C}$ that generates $\mathcal{C}$ under small colimits. Let $\mathcal{C}_{w \ge 0}$ and $\mathcal{C}_{w\le 0}$ be full subcategories generated by $(\mathcal{N}[i])_{i\ge 0}$ and $(\mathcal{N}[i])_{i\le 0}$, respectively, under small coproducts and extensions. If $\mathcal{N}[-1] \subseteq \prescript{\perp}{}{\mathcal{C}_{w \ge 0}}$ holds, then $w=(\mathcal{C}_{w \ge 0}, \mathcal{C}_{w\le 0})$ is the unique weight structure $w$ on $\mathcal{C}$ satisfying
\begin{itemize}
\item the heart of $w$ contains $\mathcal{N}$; more precisely, the heart of $w$ is the idempotent completion of the full subcategory generated by (possibly infinite) direct sums of $\mathcal{N}$;
\item the full subcategory $\mathcal{C}_{w \ge 0}$ is closed under small coproducts (or equivalently small colimits).
\end{itemize}
\end{pro}

This proposition allows us to extend the weight structure on the compact objects to the entire category:

\begin{cor}
  \label{cor:wtstr-cg}
  Let $\mathcal{C}$ be a compactly generated stable $\infty$-category. Assume the full subcategory $\mathcal{C}_{\omega}$ of compact objects has a weight structure $w$. Then $w$ can be extended to a weight structure on $\mathcal{C}$ by taking $\mathcal{N}=\mathcal{C}_{\omega}^{\heartsuit_{w}} $ in Proposition \ref{pro:wtstr-from-neg}. Moreover, it is the unique weight structure on $\mathcal{C}$ extending $w$ on $\mathcal{C}_{\omega}$ such that $\mathcal{C}_{w \ge 0}$ is closed under small colimits.
\end{cor}

For our purposes, we state the following useful instance:

\begin{cor}
  \label{cor:wtstr-cg-from-negative}
  Let $\mathcal{C}$ be a compactly generated stable $\infty$-category. Assume there is a small negative full subcategory $\mathcal{N}$ consisting of compact objects, and it generates $\mathcal{C}$ under small colimits and negative shifts.
\begin{enumerate}
\item There is a unique weight structure $w$ on $\mathcal{C}$ satisfying
  \begin{itemize}
  \item the heart of $w$ contains $\mathcal{N}$;
    
  \item $\mathcal{C}_{w \ge 0}$ is closed under small colimits.
  \end{itemize}
  
\item The weight structure $w$ in $(1)$ satisfies
  \begin{enumerate}
  \item the heart of $w$ is the idempotent completion of the full subcategory generated by $\mathcal{N}$ under directed sums;
    
  \item the weight structure $w$ restricts to the weight structure on $\mathcal{C}_{\omega}$ whose heart contains $\mathcal{N}$;
    
  \item given $X \in \mathcal{C}$, then $X \in \mathcal{C}_{w \ge 0}$ if and only if $X \in \mathcal{N}[i]^{\perp}$ for all $i<0$;
    
  \item given $X \in \mathcal{C}_{\omega}$, then $X \in \mathcal{C}_{w\le 0}$ if and only if $X \in \prescript{\perp}{}{\mathcal{N}[i]}$ for all $i>0$.
  \end{enumerate}
  
\item Assume $\mathcal{C}$ admits a symmetric monoidal structure whose tensor unit is in $\mathcal{C}^{\heartsuit_w}$. If for any pair $(X,Y)$ of objects in $\mathcal{N}$, the tensor product $X \otimes Y $ is in the heart, then $w$ is compatible with the monoidal structure.
\end{enumerate}
\end{cor}
\begin{proof}
  \begin{enumerate}
  \item By our assumption, $\mathcal{N}[-1]^{\perp}$ contains all $\mathcal{N}[i]$ for $i\ge 0$, and it is closed under extensions and small coproducts (because of compactness of objects in $\mathcal{N}$); thus $\mathcal{N}$ satisfies the assumption in Proposition \ref{pro:wtstr-from-neg}, we apply it to get a desired weight structure on $\mathcal{C}$.
    
  \item The part (a) is clear. For part (b), applying Proposition \ref{pro:bdwt-from-heart} to $\mathcal{N}$, we get the unique bounded weight structure on $\mathcal{C}_{\omega}$ whose heart contains $\mathcal{N}$. It can be extended to $\mathcal{C}$ by using Corollary \ref{cor:wtstr-cg}. This new weight structure on $\mathcal{C}$ agrees with $w$ due to the uniqueness in ($1$). It remains to prove (c) and (d). There are obvious inclusions
  \[
\mathcal{C}_{w \ge 0} \subseteq (\mathcal{N}[-i])^{\perp}, \quad \mathcal{C}_{w \le 0} \subseteq \prescript{\perp}{}{\mathcal{N}[i]},
\]
for all $i>0$. Now we prove the converse direction. For (c), it is clear: $\prescript{\perp}{}{\{X\}}$ is closed under small coproducts and extensions, and then the assumption on $X$ implies $\mathcal{C}_{w\le -1} \subseteq \prescript{\perp}{}{\{X\}}$; equivalently, we have $X \in \mathcal{C}_{w \ge 0}$ . The similar argument works to prove (d) when $X$ is compact.

\item We need to prove that $\mathcal{C}_{w\le 0}$ (resp. $\mathcal{C}_{w\ge 0}$) is closed under the tensor product. The proofs for them are similar. Let's prove for $\mathcal{C}_{w\le 0}$ (for $\mathcal{C}_{w\ge 0}$, one needs to notice that $\mathcal{C}_{w \ge 0}$ is stable under small coproducts). We define $\mathcal{D}$ as the full subcategory of $\mathcal{C}$ spanned by those $X \in \mathcal{C}$ such that $\mathcal{C}_{w\le 0} \otimes X \subseteq \mathcal{C}_{w\le 0}$. Then it suffices to prove $\mathcal{C}_{w \le 0} \subseteq \mathcal{D}$. To this end, we need to show $\mathcal{D}$ is closed under coproducts, extensions and contains $\mathcal{N}[i] $ for all $i\le 0$. As $\mathcal{C}_{w\le 0}$ is closed under coproduct and extensions, it remains to show $\mathcal{N}[i] \subseteq \mathcal{D}$ for all $i \le 0$. For this, we define a full subcategory $\mathcal{D}'$ of $\mathcal{C}$ in a way similar to $\mathcal{D}$, that is, $X \in \mathcal{D}'$ if $X \otimes \mathcal{N}[i] \subseteq \mathcal{C}_{w\le 0}$ for $i\le 0$. By our condition in the statement, we have inclusions $\mathcal{N}[j] \subseteq \mathcal{D}'$ for all $j\le 0$. It is easy to see that $\mathcal{D}'$ is closed under small coproducts and extensions as well. This implies $\mathcal{C}_{w\le 0} \subseteq \mathcal{D}'$, as desired.
\end{enumerate}
\end{proof}

\begin{cor}
  \label{cor:cg-wtex}
  Let $F \colon \mathcal{C}_{1} \to \mathcal{C}_2$ be a functor in $\prlc$. Assume, for $i=1,2$, there is a small full subcategory $\mathcal{N}_{i}$ satisfying the conditions in Corollary \ref{cor:wtstr-cg-from-negative}. Let $w_i$ be the weight structure obtained in Corollary \ref{cor:wtstr-cg-from-negative} for $i=1,2$. Then the following are equivalent:
  \begin{enumerate}
  \item $F$ is weight-exact;
  \item after restricting to compact objects, $F \colon \mathcal{C}_{1,\omega} \to \mathcal{C}_{2,\omega}$ is weight-exact;
    
  \item $F$ sends $\mathcal{N}_{1}$ into the heart of $w_{2}$.
\end{enumerate}
\end{cor}
\begin{proof}
We prove the non-obvious direction \imp{3}{1}. In fact, this follows from the fact that $F$ preserves small colimits and extensions.
\end{proof}


\section{Weight Structures on Rigid Analytic Motives}
In the section, we are interested in weight structures on motives. Our goal is to construct a weight structure on the whole category of rigid analytic motives over a complete non-archimedean field. For this, we first recall the $\infty$-category of motive in \S \ref{sec:rec-motive} and recall a relationship between analytic motives on the generic fibers and algebraic motives on the special fiber (see Proposition \ref{pro:rigmot-as-module}). In \S \ref{subsec:gal-descent}, we prove categorical formulae to relate general rigid analytic motives with those with good reduction by Galois descent. Finally, we construct the weight structure on rigid analytic motives over fields in \S \ref{subsec:wtstrproof}.

\subsection{Stable $\infty$-Categories of Motives}
\label{sec:rec-motive}

Let $K$ be a complete non-archimedean field with the residue field $k$ of characteristic $p>0$. We denote by $\mathcal{O}_K$ the ring of integers. Then we have the following $\infty$-categories of motives:
\begin{itemize}
\item the $\infty$-category of \myemph{rigid analytic motives} $\rigmot(K):= \mathbf{RigSH}\ets(\spa(K, \mathcal{O}_K), \Q)$ over $K$ defined in \cite[Definition 2.1.15]{AGV22};
  
\item the $\infty$-category of \myemph{formal motives} $\fagmot(\mathcal{O}_K) := \mathbf{FSH}\ets(\spf(\mathcal{O}_K), \Q)$ over $\mathcal{O}_K$ defined in \cite[Definition 3.1.3]{AGV22};
  
\item the $\infty$-category of \myemph{algebraic motives} $\agmot(k):= \fagmot(\spec(k), \Q)$ over $k$, where $\spec(k)$ is regarded as a formal scheme in the obvious way.
\end{itemize}
As shown in \cite{ayothesis1,ayothesis2,Ayo14a,ayorigmot,AGV22}, these are compactly generated stable symmetric monoidal monoidal $\infty$-categories.

A theorem of Ayoub, as stated in \cite[Corollaire 1.4.24 \& 1.4.29]{ayorigmot} or \cite[Theorem 3.1.10]{AGV22}, yields a symmetric monoidal functor, called the \myemph{Monsky-Washnitzer functor},
\begin{equation}
  \label{functor:WM}
\xi \colon \agmot(k) \simeq \fagmot(\mathcal{O}_K) \xrightarrow{(-)_{\eta}} \rigmot(K)
\end{equation}
in $\calg(\prlc)$, where $(-)_{\eta}$ is induced by the rigid generic fiber functor of formal schemes. We denote its right adjoint by
\[
\chi\colon \rigmot(K) \to \agmot(k).
\]
There is a remarkable way to relate motives on the rigid generic fiber with motives on the special fiber:

\begin{pro}
  \label{pro:rigmot-as-module}
   Consider the commutative algebra object $\chi {\mathbbm 1} \in \calg(\agmot(k))$, where ${\mathbbm 1}$ is the tensor unit in $\rigmot(K)$. Then the Monsky-Washnitzer functor $\xi$ factors through the free module functor $\mathrm{Free}\colon \agmot(k) \to \module{\chi {\mathbbm 1}}(\agmot(k))$; in other words, there is a functor
  \[
\tilde{\xi}\colon \module{\chi {\mathbbm 1}} (\agmot(k)) \to \rigmot(K)
\]
with an equivalence $\tilde{\xi} \circ \mathrm{Free} \simeq \xi$. Moreover, $\tilde{\xi}$ is fully faithful.
\end{pro}
\begin{proof}
The first assertion is clear. For the full faithfulness, see \cite[Theorem 3.3.3]{AGV22}.
\end{proof}

\begin{df}
  \label{df:rigmotgr}
  We let $\rigmot_{\mathrm{gr}}(K)$ denote the essential image of $\tilde{\xi}$ in Proposition \ref{pro:rigmot-as-module} and refer to it as the $\infty$-category of \myemph{rigid analytic motives with good reduction}.
\end{df}

With this notation, we have an equivalence of $\infty$-categories $\rigmot_{\mathrm{gr}}(K) \simeq \module{\chi {\mathbbm 1}}(\agmot(k))$.

\begin{rmk}
  \label{rmk:ssred-in-gr}
  The $\infty$-category $\rigmot_{\mathrm{gr}}(K)$ is also the full subcategory of $\rigmot(K)$ generated by $M(\mathcal{X}_{\eta})$ under colimits, where $\mathcal{X}$ runs through smooth formal schemes over $\mathcal{O}_K$. So this category contains more motives than those from rigid varieties with good reduction. In fact, it also contains motives associated to rigid analytic spaces with semi-stable reductions; see also \cite[Proposition 3.29]{BKV25} for more general examples.
\end{rmk}


\subsection{Galois Equivariant Rigid Analytic Motives}
\label{subsec:gal-descent}

\begin{cons}
  \label{cons:gal-inv}
  Let $L/K$ be a finite Galois extension of complete non-archimedean fields with the Galois group $\gal (L/K)$. It is clear that $\rigmot(L)$ admits a canonical action by $\gal (L/K)$ in the sense of Definition \ref{df:ht-fix-pt}. More precisely, we have a functor obtained by taking the homotopy coherent nerve\footnote{Indeed, taking the homotopy coherent nerve gives $\subss{B} \gal (L/K) \to \bigcatinf$. Then using \cite[Corollary 2.4.1.9]{HA} and the fact that $\rigmot(K) \in \calg(\prl)$, we get this functor.}
\begin{align}
  \label{eq:df-gal-action}
  \begin{split}
A_{L/K}\colon  B_{\bullet} \gal (L/K)&\to \calg(\prl)\\
  *&\mapsto \rigmot(L)\\
  (\sigma \in \gal (L/K))&\mapsto \left( \sigma^{*} \colon \rigmot(L) \to \rigmot(L) \right)
                             \end{split}
\end{align}
Taking homotopy fixed points, we get a new category $\rigmot(L)^{\htcat \gal (L/K)}$ in $\calg(\prl)$. We will refer to objects in this new category as \myemph{$\boldsymbol{\gal (L/K)}$-equivariant rigid analytic motives} over $L$.
\end{cons}

In the rest of this subsection, we fix a finite Galois extension $L/K$ of complete non-archimedean fields with the Galois group $\gal (L/K)$. Let $k_L/k$ be the corresponding extension of residue fields, and assume they are perfect. We will denote by $e\colon \spa(L) \to \spa(K)$ and $\bar{e}\colon \spec(k_L) \to \spec(k)$ the corresponding morphisms in the geometric world.

The next lemma shows that the Galois action restricts to an action on the subcategory of those with good reduction.

\begin{lem}[{\cite[Proposition 3.1.13]{AGV22}}]
  \label{lem:gen-fib-pullback-sharp}
 Let $e\colon \spa(L) \to \spa(K)$ and
  $\bar{e}\colon \spec k_L \to \spec k$ be the canonical
  maps. Consider the two diagrams below:
  \begin{figure}
[H]\centering
\begin{tikzcd}
\agmot(k_{L}) \arrow[r, "\xi_L"]                        & \rigmot(L)                   \\
\agmot(k) \arrow[r, "\xi_K"] \arrow[u, "\bar{e}^*"] & \rigmot(K) \arrow[u, "e^*"']
\end{tikzcd}
\quad in $\calg(\prl_{\omega})$,
\end{figure}\noindent
\quad
\begin{figure}
[H]\centering
\begin{tikzcd}
\agmot(k_{L}) \arrow[r, "\xi_L"] \arrow[d, "\bar{e}_{\sharp}"'] & \rigmot(L) \arrow[d, "e_{\sharp}"] \\
\agmot(k) \arrow[r, "\xi_K"]                                & \rigmot(K)                        
\end{tikzcd}
\quad in $\bigcatinf$.
\end{figure}\noindent
The first diagram is commutative (up to homotopy) and the second diagram is commutative if $L/K$ is unramified. In particular, we have restrictions
\begin{align*}
  e^{*}\colon \rigmot_{\mathrm{gr}}(K)&\to \rigmot_{\mathrm{gr}}(L)\\
  e_{\sharp}\colon \rigmot_{\mathrm{gr}}(L)&\to \rigmot_{\mathrm{gr}}(K),
\end{align*}
where $e_{\sharp}$ is well-defined if $L/K$ is unramified.
\end{lem}

The Galois group $\gal (L/K)$ acts on $\rigmot_{\mathrm{gr}}(L)$ naturally thanks to Lemma \ref{lem:gen-fib-pullback-sharp}. In other words, we have a functor
\begin{align*}
  A_{L/K, \mathrm{gr}} \colon \subss{B} \gal (L/K) &\to \calg(\prl)\\
  *&\mapsto \rigmot_{\mathrm{gr}}(L)\\
  (\sigma \in \gal (L/K))&\mapsto (\sigma^{*} \colon \rigmot_{\mathrm{gr}}(L) \to \rigmot_{\mathrm{gr}}(L)).
\end{align*}
The natural inclusion $\rigmot_{\mathrm{gr}}(L) \hookrightarrow \rigmot(L)$ defines a natural transformation between these actions $\iota \colon A_{L/K, \mathrm{gr}} \to A_{L/K}$, where $A_{L/K}$ is defined in (\ref{eq:df-gal-action}). Moreover, this makes $\rigmot_{\mathrm{gr}}(L)$ be a $G$-subcategory of $\rigmot(L)$ in the sense of Definition \ref{df:G-subcat}. Thus, we deduce from Corollary \ref{cor:grp-inv-preserve-ff} that:

\begin{cor}
  \label{cor:grmot-gal-action}
The inclusion functor induces a fully faithful functor
  \[
\rigmot_{\mathrm{gr}}(L)^{\htcat \gal (L/K)} \hookrightarrow \rigmot(L)^{\htcat \gal (L/K)}.
\]
\end{cor}
\subsubsection{The Galois Descent}
\label{sec:galo-invar-motiv}
Using the Galois descent of rigid motives, we can relate motives over $K$ to motives over the Galois extension field $L$.

\begin{pro}
  \label{pro:galois-descent-mot}
  The canonical functor
  \begin{equation}
    \label{eq:galois-descent}
\tilde{e}^{*}\colon \rigmot(K) \xrightarrow{\simeq} \rigmot(L)^{\htcat \gal (L/K)}    ,
  \end{equation}
induced by $e^{*}\colon \rigmot(K) \to \rigmot(L)$, is an equivalence in $\calg(\prl)$. 
\end{pro}
\begin{proof}
  Since the forgetful functor $\calg(\prl) \to \prl$ is limit-preserving and conservative by Proposition \ref{pro:clim-calgpr}, and limits in $\prl$ are computed as $\infty$-categories by Proposition \ref{pro:lim-prl-prr}, it suffices to prove \eqref{eq:galois-descent} is an equivalence in $\bigcatinf$. In fact, this is a direct consequence of the \'{e}tale descent for
$\rigmot(-)$. We can write $B_{\bullet} \gal (L/K)$ as
a colimit of a geometric realization:
\[
B_{\bullet} \gal (L/K) \simeq \colim_{\cnerve(\Delta)^{\opp}} S_{\bullet},
\]
where $S_n= \gal (L/K)^n$, see Lemma \ref{lem:BG-georeal}. It follows that $\rigmot(L)^{\htcat \gal
  (L/K)}$ can be identified with the limit of a cosimplicial object
$X^{\bullet}$ of $\prl$, where $X^n$ is a limit of the induced diagram
\[
S_n \to B_{\bullet} \gal (L/K)  \to \prl.
\]
Since $S_n$ is discrete, we have $X^n\simeq
\prod_{\sigma_1,\dots,\sigma_n \in \gal (L/K)} \rigmot(L) $. In other words, we have
\[
  \rigmot(L)^{\htcat \gal (L/K)} \simeq \lim \left(
\begin{tikzcd}
\rigmot(L) \arrow[r, shift left] \arrow[r, shift right] & {\displaystyle\prod_{\gal(L/K)}} \rigmot(L) \arrow[r] \arrow[r, shift right] \arrow[r, shift left] & \cdots
\end{tikzcd}
  \right)
\]
On the other hand, this is also the \v{C}ech nerve of the canonical
map
\[
e^{*}\colon \rigmot(L) \to \rigmot(K).
\]
We therefore deduce from the \'{e}tale descent (see \cite[Theorem 2.3.4 \& Remark 2.3.5]{AGV22}).
\end{proof}

We next study the restriction of the equivalence (\ref{eq:galois-descent}) to $\rigmot_{\mathrm{gr}}(L)^{\htcat \gal (L/K)}$.

\begin{notation}
  We define $\rigmot_{L \text{-}\mathrm{gr}}(K)$ by the pullback diagram
  \begin{equation}
    \label{eq:Lgr}
\begin{tikzcd}
\rigmot_{L \text{-}\mathrm{gr}}(K) \arrow[d] \arrow[r] & \rigmot_{\mathrm{gr}}(L) \arrow[d, hook] \\
\rigmot(K) \arrow[r]                                   & \rigmot(L)                              
\end{tikzcd}
  \end{equation}
in $\calg(\prl)$.
\end{notation}

\begin{rmk}
  \label{rmk:Lgr-subcat}
  The underlying functor $\rigmot_{L \text{-}\mathrm{gr}}(K) \to \rigmot(K)$ is fully faithful and we can identify (up to equivalence) $\rigmot_{L \text{-}\mathrm{gr}}(K)$ with the full subcategory of $\rigmot(K)$ spanned by those $M$ satisfying $e^{*}M \in \rigmot_{\mathrm{gr}}(L)$, where $e\colon \spa(L) \to \spa(K)$ is the structure map. In particular, $\rigmot_{\mathrm{gr}}(K)$ is a full subcategory of $\rigmot_{L \text{-}\mathrm{gr}}(K)$ thanks to Lemma \ref{lem:gen-fib-pullback-sharp}.
\end{rmk}

We will use the identification in Remark \ref{rmk:Lgr-subcat} implicitly from now on.

\begin{pro}
  \label{pro:galois-descent-mot-gdrd}
  The canonical functor $\rigmot_{L \text{-}\mathrm{gr}}(K) \to \rigmot_{\mathrm{gr}}(L)$ factors through the canonical functor $\rigmot_{\mathrm{gr}}(L)^{\htcat \gal (L/K)}\to \rigmot_{\mathrm{gr}}(L)$, and induces an equivalence
\begin{figure}
[H]\centering
\begin{tikzcd}
                                                                          & \rigmot_{\mathrm{gr}}(L)^{\htcat \gal (L/K)} \arrow[d] \\
\rigmot_{L \text{-}\mathrm{gr}}(K) \arrow[r] \arrow[ru, "\simeq", dashed] & \rigmot_{\mathrm{gr}}(L)                              
\end{tikzcd}
\end{figure}\noindent
in $\calg(\prl)$. Moreover, we have a commutative diagram (up to homotopy)
\begin{figure}
[H]\centering
\begin{tikzcd}
\rigmot_{L \text{-}\mathrm{gr}}(K) \arrow[r, "\simeq", dashed] \arrow[d, hook'] & \rigmot_{\mathrm{gr}}(L)^{\htcat \gal (L/K)} \arrow[r] \arrow[d, hook'] & \rigmot_{\mathrm{gr}}(L) \arrow[d, hook'] \\
\rigmot(K) \arrow[r, "\tilde{e}^{*}"',"\simeq"]                                                  & \rigmot(L)^{\htcat \gal (L/K)} \arrow[r]                                & \rigmot(L)                               
\end{tikzcd}
\end{figure}\noindent
in $\calg(\prl)$, where the middle fully faithful functor is the one in Corollary \ref{cor:grmot-gal-action}.
\end{pro}
\begin{proof}
  Using the identification in Remark \ref{rmk:Lgr-subcat}, the pullback functor $e^{*}\colon \rigmot(K) \to \rigmot(L)$ restricts to a functor
\[
e^{*}_{\mathrm{gr}}\colon \rigmot_{L \text{-}\mathrm{gr}}(K) \to \rigmot_{\mathrm{gr}}(L),
\]
which is also $\gal (L/K)$-invariant. Therefore, we get a functor
\[
\rigmot_{L \text{-}\mathrm{gr}}(K) \to
\rigmot_{\mathrm{gr}}(L)^{\htcat \gal (L/K)}
\]
in $\calg(\prl)$ by the definition of $\rigmot_{\mathrm{gr}}(L)^{\htcat \gal (L/K)}$. In fact, it is just the restriction of the equivalence (\ref{eq:galois-descent}). This gives the commutative diagram in the last assertion except for the dotted equivalence. 

Since the forgetful functor $\calg(\prl) \to \prl$ is conservative, it suffices to show this restriction is an equivalence in $\prl$. As a restriction of an equivalence, it is clearly fully faithful. It remains to prove the essential surjectivity. Let $M $ be an
object of $\rigmot_{\mathrm{gr}}(L)^{\htcat \gal (L/K)}$. Using the equivalence (\ref{eq:galois-descent}) in Proposition \ref{pro:galois-descent-mot}, we can find a (unique) $N \in \rigmot(K)$ such that $\tilde{e}^{*}(N) \simeq M$. The commutativity of the diagram shows $e^{*}N \in \rigmot_{\mathrm{gr}}(L)$, i.e., $ N \in \rigmot_{L \text{-}\mathrm{gr}}(K)$. Therefore, this proves $M \simeq e^{*}N \in \rigmot_{\mathrm{gr}}(L)^{\htcat \gal (L/K)}$.
\end{proof}

\subsubsection{The Compact Generation}
\label{sec:compact-generators}

Now we study the compact generation of $\rigmot_{\mathrm{gr}}(L)^{\htcat \gal (L/K)}$ and give an explicit description for its compact generators.

\begin{lem}
  \label{lem:left-adjoint-G-inv}
  Let $e\colon \spa(L) \to \spa(K)$ be the structure morphism.
  \begin{enumerate}
  \item Restricting $e_{*} \colon \rigmot(L) \to \rigmot(K)$ to $\rigmot_{\mathrm{gr}}(L)$ gives a left adjoint
    \[
e_{*} \colon \rigmot_{\mathrm{gr}}(L) \to \rigmot_{L-\mathrm{gr}}(K)
\]
of the pullback functor $e^{*}_{\mathrm{gr}} \colon \rigmot_{L \text{-}\mathrm{gr}}(K)\to \rigmot_{\mathrm{gr}}(L)$.

\item Under the equivalence $\rigmot_{L \text{-}\mathrm{gr}}(K) \simeq \rigmot_{\mathrm{gr}}(L)^{\htcat \gal (L/K)}$ in Proposition \ref{pro:galois-descent-mot-gdrd}, the restricted functor in $(1)$ is both left and right adjoint to the forgetful functor $\iota_L\colon
\rigmot_{\mathrm{gr}}(L)^{\htcat \gal (L/K)} \to \rigmot_{\mathrm{gr}}(L)$.
\end{enumerate}
\end{lem}
\begin{proof}
Let $L_0$ be the finite unramified extension of $K$ inside $L$ with the residue field $k_L$. Then $e \colon \spa(L) \to \spa(K)$ factors as $\spa(L) \xrightarrow{t} \spa(L_0) \xrightarrow{e_0} \spa(K)$. Since $t^{*}$ is an equivalence by \cite[Proposition 3.23]{BKV25}, we have a commutative diagram
\begin{equation}
  \label{eq:lowstargr}
\begin{tikzcd}
                                                                                       & \rigmot(L) \arrow[d, "t_*"]                                          \\
\agmot(k_L) \arrow[d, "\bar{e}_{\sharp}"'] \arrow[ru, "\xi_L"] \arrow[r, "\xi_{L_0}"'] & \rigmot(L_0) \arrow[d, "{e_{0,\sharp}\simeq e_{0,!}\simeq e_{0,*}}"] \\
\agmot(k) \arrow[r, "\xi_K"]                                                           & \rigmot(K)                                                          
\end{tikzcd}
\end{equation}
where the square is commutative thanks to Lemma \ref{lem:gen-fib-pullback-sharp}. Then we have
\begin{equation}
  \label{eq:underlying-cptgen}
  e^{*}e_{*}\xi_L \simeq e^{*} \xi_K \bar{e}_{\sharp} \simeq \xi_L \bar{e}^{*} \bar{e}_{\sharp}.
\end{equation}
This proves ($1$). For part ($2$), the forgetful functor $\iota_{L}$ corresponds to the pullback functor $e^{*}$ under the identification. Thus, we can deduce from ($1$).
\end{proof}

\begin{notation}
  \label{not:nm-functor}
  From now on, we let $\Nm_L\colon \rigmot_{\mathrm{gr}}(L) \to \rigmot_{\mathrm{gr}}(L)^{\htcat \gal (L/K)}$ denote the left adjoint of the forgetful functor $\iota_{L}$. We call it the \myemph{norm functor}\footnote{This is suggested due to Proposition \ref{pro:proper-sm-under-lad} below. See also \cite[\S 6.1.6]{HA} and \cite{HL13} for additional information.}.
\end{notation}

\begin{pro}
  \label{pro:cpt-gen-gal-inv-gd}
  The $\infty$-category $\rigmot_{\mathrm{gr}}(L)^{\htcat \gal (L/K)}$
  is compactly generated. A set of compact generators is given, up to negative shifts and Tate twists, by $\Nm_L \xi_LM(X)$, where $X$ runs
  through proper smooth algebraic varieties over $k_L$.
\end{pro}
\begin{proof}
Recall that the underlying $\infty$-category $\rigmot_{\mathrm{gr}}(L)^{\htcat
  \gal (L/K)}$ is the limit of
\[
G_{L/K}\colon B_{\bullet} \gal (L/K) \to \calg(\prl) \to \prl.
\]
Since each $\sigma^{*}$ is an equivalence, it is also in $\prr$. In other words, we have a commutative diagram
\begin{figure}
[H]\centering
\begin{tikzcd}
B_{\bullet} \gal(L/K) \arrow[r, "G_{L/K}"] \arrow[d, "G_{L/K}"'] & \prl \arrow[d]    \\
\prr \arrow[r]                                                   & \bigcatinf
\end{tikzcd}
\end{figure}\noindent
So the limit is essentially computed in $\bigcatinf$ (see Proposition \ref{pro:lim-prl-prr}), and it can be computed in $\prr$ as well.

Under the equivalence in Proposition \ref{pro:rigmot-as-module} and using \cite[Proposition 15.2.3]{cd19}, we know that $\rigmot_{\mathrm{gr}}(L)$ is compactly generated by those motives of the form $\xi_LM(X)$, where $X$ runs through proper smooth varieties over $k_L$. Therefore, we conclude from Proposition \ref{pro:lim-in-prrc}. 
\end{proof} 

We give an explicit formula for the $\gal (L/K)$-equivariant rigid analytic motives with good model.

\begin{pro}
  \label{pro:proper-sm-under-lad}
 For every proper smooth algebraic
  variety $X$ over the residue field $k_L$ of $L$, we have an
  isomorphism
  \[
\iota_{L}\Nm_L(\xi_L M(X)) \simeq \bigoplus_{e_{L/K}} \xi_L M(X),
\]
where $e_{L/K}$ is the ramification index of $L/K$. In particular, the forgetful functor
\[
\iota_{L} \colon \rigmot_{\mathrm{gr}}(L)^{\htcat \gal (L/k)} \to \rigmot_{\mathrm{gr}}(L)
\]
is in $\calg(\prlc)$.
\end{pro}
\begin{proof}
  We use notations in the proof of Lemma \ref{lem:left-adjoint-G-inv}. By (\ref{eq:underlying-cptgen}), we have
\[
\iota_L \Nm_L (\xi_L M(X)) \simeq e^{*}e_{*}(\xi_LM(X)) \simeq \xi_L \bar{e}^{*} \bar{e}_{\sharp}M(X).
\]
Thus, it suffices to prove $\bar{e}^{*} \bar{e}_{\sharp} M(X) \simeq
\oplus_{e_{L/K}} M(X)$. For this, we consider the following Cartesian diagrams of schemes:
\begin{figure}
[H]\centering
\begin{tikzcd}
\tilde{X} \arrow[r, "p"] \arrow[d, "\tilde{f}"']                   & X \arrow[d, "f"]             \\
\spec(l\otimes_k l) \arrow[r, "\tilde{e}"] \arrow[d, "\tilde{e}"'] & \spec l \arrow[d, "\bar{e}"] \\
\spec l \arrow[r, "\bar{e}"']                                      & \spec k                     
\end{tikzcd}
\end{figure}\noindent
where $\tilde{X}$ is the fiber product $X \otimes_k l$. Then $\bar{e}^{*}\bar{e}_{\sharp}M(X)\simeq M(\tilde{X})$, where, as a $k_{L}$-scheme, the structure map of $\tilde{X}$ is $\tilde{e} \circ \tilde{f}$. Now let's prove $M(\tilde{X}) $ is the $e_{L/K}$-copies of $M(X)$:
\[
M(\tilde{X} ) \simeq M(X) \otimes M(\spec(k_{L}\otimes_{k} k_{L})).
\]
Under the identification
\[
\agmot(\spec (k_{L} \otimes_k k_{L})) \simeq \agmot( \coprod_{\gal (k_{L}/k)} \spec  k_{L})
  \simeq \prod_{\gal (k_{L}/k)} \agmot(k_{L}),
\]
we know that
\begin{align*}
  M(\tilde{X}) &\simeq M(X) \otimes \left( \bigoplus_{\tau \in \gal (k_{L}/k)} \tau_{*}  {\mathbbm 1} \right) \simeq \bigoplus_{\tau \in \gal (k_{L}/k)} M(X) \otimes \tau_{*} {\mathbbm 1}\\
  &\simeq \bigoplus_{\tau \in \gal (k_{L}/k)}\tau_{*} \tau^{*} M(X) \simeq \bigoplus_{\tau \in \gal (k_{L}/k)} M(X).
\end{align*}
\end{proof}

\begin{cor}
  \label{cor:galcolim-transmap}
  Let $L/F/K$ be finite Galois extensions of complete non-archimedean fields with residue fields $k_L/k_{F}/k$ respectively. There is a fully faithful functor
  \[
\rigmot_{\mathrm{gr}}(F)^{\htcat \gal (F/K)} \hookrightarrow \rigmot_{\mathrm{gr}}(L)^{\htcat \gal (L/K)}
\]
and an equivalence
\[
\rigmot(F)^{\htcat \gal (F/K)} \xrightarrow{\simeq} \rigmot(L)^{\htcat \gal (L/K)}
\]
in $\calg(\prl)$ such that the following diagram
  \begin{figure}
[H]\centering
\begin{tikzcd}
\rigmot_{F \text{-}\mathrm{gr}}(K) \arrow[d, hook'] \arrow[r, "\simeq"] & \rigmot_{\mathrm{gr}}(F)^{\htcat \gal (F/K)} \arrow[d, "\exists", hook'] \arrow[r, hook] & \rigmot(F)^{\htcat \gal (F/K)} \arrow[d, "\simeq"] \arrow[r] & \rigmot(F) \arrow[d] \\
\rigmot_{L \text{-}\mathrm{gr}}(K) \arrow[r, "\simeq"]                  & \rigmot_{\mathrm{gr}}(L)^{\htcat \gal (L/K)} \arrow[r, hook]                             & \rigmot(L)^{\htcat \gal (L/K)} \arrow[r]                     & \rigmot(L)          
\end{tikzcd}
\end{figure}\noindent
is commutative (up to homotopy), where the right-most vertical functor is the pullback functor. Moreover, the faithful functor above is in $\calg(\prlc)$.
\end{cor}
\begin{proof}
  For finite Galois extensions $L/K$ and $F/K$, both of them admit an equivalence as Proposition \ref{pro:galois-descent-mot}. Thus, this gives a natural equivalence between $\rigmot(F)^{\htcat \gal (F/K)} $ and $\rigmot(L)^{\htcat \gal (L/K)}$, making the right-most square commutative (up to homotopy). Similarly, using equivalences in Proposition \ref{pro:galois-descent-mot-gdrd}, i.e.,
  \begin{align*}
    \rigmot_{F \text{-}\mathrm{gr}}(K) &\simeq \rigmot_{\mathrm{gr}}(F)^{\htcat \gal (F/K)}\\
    \rigmot_{L \text{-}\mathrm{gr}}(K) &\simeq \rigmot_{\mathrm{gr}}(L)^{\htcat \gal (L/K)},
  \end{align*}
  we can get the desired fully faithful functor. Now, as the leftmost, rightmost and the outer squares are commutative, the middle one is also commutative by our constructions.

  Finally, let's show the faithful functor
  \[
\rigmot_{\mathrm{gr}}(F)^{\htcat \gal (F/K)} \hookrightarrow \rigmot_{\mathrm{gr}}(L)^{\htcat \gal (L/K)}
\]
preserves compact objects. We let $\alpha\colon \spa(L) \to \spa(F)$ and $\beta\colon \spa(F) \to \spa(K)$ be obvious structure morphisms. Let $X$ be a proper smooth algebraic variety over $k_{F}$. We need to show $e^{*}\beta_{*} \xi_F M(X) \in \rigmot_{\mathrm{gr}}(L)$. We can deduce from Proposition \ref{pro:proper-sm-under-lad}:
\begin{equation}
\label{eq:gal-inv-wt}
\begin{aligned}
  e^{*} \beta_{*} \xi_FM(X) &\simeq \alpha^{*} \beta^{*} \beta_{*} \xi_F M(X) \\
                            &\simeq \alpha^{*}\left( \bigoplus_{e_{F/K}} \xi_FM(X) \right) \\
  &\simeq \bigoplus_{e_{F/K}} \xi_L M(X \times_{k_F}k_L) \in \rigmot_{\mathrm{gr}}(L).
\end{aligned}
\end{equation}
\end{proof}

\subsubsection{The Sheafification}
\label{sec:gluing}
We conclude this subsection by presenting the categorical formulation of the fact that every rigid analytic motive over $K$ has potentially good reduction. This gives the precise form of \cite[Theorem 3.7.1]{AGV22} in the case where the base is a field.

Under the equivalence
\[
\rigmot_{\mathrm{gr}}(L)^{\htcat \gal (L/K)} \simeq \rigmot_{L \text{-}\mathrm{gr}}(K),
\]
we have a fully faithful functor
\[
\rigmot_{\mathrm{gr}}(L)^{\htcat \gal (L/K)} \hookrightarrow \rigmot(K)
\]
for every finite Galois extension $L/K$. Using Corollary \ref{cor:galcolim-transmap} and taking colimits, we have a functor
\begin{equation}
  \label{eq:mot-union-gr}
  \colim_{\substack{L/K\\\text{finite Galois}}} \rigmot_{\mathrm{gr}}(L)^{\htcat \gal (L/K)} \to \rigmot(K)
\end{equation}
in $\calg(\prlc)$. Now we show this is indeed an equivalence.

\begin{pro}
  \label{pro:rigmot-cptgen-pgd}
The functor (\ref{eq:mot-union-gr}) is an equivalence in $\calg(\prlc)$. In particular, we have an equivalence
\[
\rigmot(K)_{\omega} \simeq \colim \rigmot_{\mathrm{gr}}(L)^{\htcat \gal (L/K)}_{\omega}.
\]
\end{pro}
\begin{proof}
  Since the forgetful functor $\calg(\prlc) \to \prlc$ is conservative and preserves this filtered colimit (see Proposition \ref{pro:clim-calgpr}), we prove this is an equivalence in $\prlc$. It suffices to prove the restriction
  \[
\colim \rigmot_{\mathrm{gr}}(L)^{\htcat \gal (L/K)}_{\omega} \to \rigmot(K)_{\omega}
\]
is an equivalence of $\infty$-categories, where the colimit on the left-hand side is taken in $\bigcatinf$ (see \cite[Proposition 5.5.7.10 \& Proposition 5.5.7.11]{HTT}). Firstly, since this functor is obtained by taking the colimit of fully faithful functors
\[
\rigmot_{\mathrm{gr}}(L)^{\htcat \gal (L/K)}_{\omega} \simeq \rigmot_{L \text{-}\mathrm{gr}}(K)_{\omega} \hookrightarrow \rigmot(K)_{\omega},
\]
it is also fully faithful (\cite[Proposition 2.1]{HRS25}). It remains to show it is essentially surjective. By \cite[Theorem
2.5.34]{ayorigmot} or \cite[Proposition 3.7.17]{AGV22}, for every object $M$ in $\rigmot(K)_{\omega}$, we can find a finite separable extension $\tilde{L}/K$ such that $M_{\tilde{L}}$ lies in $\rigmot_{\tilde{L} \text{-}\mathrm{gr}}(K)$. So we can take $L$ as the Galois closure of $\tilde{L}/K$, and then $M$ lies in $\rigmot_{L \text{-}\mathrm{gr}}(K)$. In particular, $M$ is given by
\[
\Delta^0 \to \rigmot_{\mathrm{gr}}(L)^{\htcat \gal (L/K)}_{\omega} \to \colim \rigmot_{\mathrm{gr}}(F)^{\htcat \gal (F/K)}_{\omega} \to \rigmot(K)_{\omega}.
\]
This proves the essential surjectivity and completes the proof.
\end{proof}


\subsection{A Weight Structure on Rigid Analytic Motives}
\label{subsec:wtstrproof}
From now on, we fix a complete non-archimedean field $K$ with perfect residue field $k$. The goal of this subsection is to construct a weight structure on $\rigmot(K)$ that extends the weight structure on $\rigmot_{\mathrm{gr}}(K)$ as defined in \cite{BGV25}.

Recall from \cite[Proposition 4.25]{BGV25} the heart of the bounded weight structure on $\rigmot_{\mathrm{gr}}(K)_{\omega}$ is the idempotent completion of the full subcategory spanned by those rigid analytic motives of the form $\xi_{K}M(X)$, where $X$ runs through a proper smooth algebraic variety over $k$. In particular, the full subcategory of $\rigmot_{\mathrm{gr}}(K)$ spanned by all motives of the form $\xi_{K}M(X)$ with $X$ proper smooth over $k$ is negative, by Proposition \ref{pro:bdwt-from-heart}.

\begin{rmk}
  \label{rmk:wtstr-gr-rel}
  In fact, the bounded weight structure on $\rigmot_{\mathrm{gr}}(K)_{\omega}$ constructed in \cite{BGV25} is uniquely characterized by the property that the Monsky-Washnitzer functor $\xi \colon \agmot(k)_{\omega} \to \rigmot_{\mathrm{gr}}(K)_{\omega}$ is weight-exact, where $\agmot(k)$ is equipped with the Chow weight structure as defined in \cite[\S 6]{Bon10a} (see also \cite{Heb11, Bon14} for the case of a general base). More generally, if $S$ is a rigid analytic space with good reduction, a similar argument shows such a bounded weight structure on $\rigmot_{\mathrm{gr}}(S)_{\omega}$ exists as well.

  On the other hand, if one only assume $S$ has semistable reduction, then the argument breaks down. The reason is that, in the case, the construction of such a bounded weight structure relies essentially on the six-functor formalism of motives, whereas the Chow weight structure is only partially compatible with the six-functor formalism of algebraic motives; see \cite{Heb11} for further discussion.
\end{rmk}

\begin{lem}
  \label{lem:wt-str-galinv}
  There is a weight structure $w_{L/K}$ compatible with monoidal structure on $\rigmot_{\mathrm{gr}}(L)^{\htcat \gal (L/K)}$ such that
  \[
\rigmot_{\mathrm{gr}}(L)^{\htcat \gal (L/K)} \to \rigmot_{\mathrm{gr}}(L)
\]
is weight-exact, where the weight structure on the target is the one defined in \cite[Proposition 4.25]{BGV25}. Moreover, equipped with such weight structures, the fully faithful functor in Corollary \ref{cor:galcolim-transmap} is weight-exact.
\end{lem}
\begin{proof}
  We construct this weight structure using Corollary \ref{cor:wtstr-cg-from-negative}. By Proposition \ref{pro:cpt-gen-gal-inv-gd}, we only need to check that, for any $X$, $Y$ proper smooth varieties over $k_{L}$, we have
  \[
\hom \left( \Nm_L \xi_L M(X), \Nm_{L} \xi_L M(Y)[i] \right)=0
\]
for $i>0$. By the adjunction and Proposition \ref{pro:proper-sm-under-lad}, it follows from the negativity of the heart of the weight structure on $\rigmot_{\mathrm{gr}}(L)$ (see also the discussion above Remark~\ref{rmk:wtstr-gr-rel}).

For the compatibility with monoidal structure, it suffices to notice that the adjunction $\Nm_L \dashv \iota_{L}$ satisfies the projection formula by their definitions; see Lemma \ref{lem:left-adjoint-G-inv}. Therefore, we have
\begin{equation}
  \label{eq:wtmonoidal}
  \begin{aligned}
    \Nm_L \xi_L M(X)&\otimes \Nm_L \xi_L M(Y)\\
                    &\simeq  \Nm_L \left( \xi_L M(X)\otimes \iota_L \Nm_L \xi_LM(Y) \right)\\
    &\simeq \bigoplus_{e_{L/K}} \Nm_L \xi_L M(X\times_{k_{L}} Y),
  \end{aligned}
\end{equation}
where the last isomorphism is also due to Proposition \ref{pro:proper-sm-under-lad}, and it is clearly in the heart by our construction. The weight-exactness of the forgetful functor $\iota_L$ is an immediate result of  Corollary \ref{cor:cg-wtex}. For the last assertion, one can check using (\ref{eq:gal-inv-wt}).
\end{proof}

We now give the weight structure on $\rigmot(K)$ by gluing the weight structure given in Lemma \ref{lem:wt-str-galinv} along the colimit in Proposition \ref{pro:rigmot-cptgen-pgd}.

\begin{thm}
  \label{thm:wt-str-rigmotcat}
  There is a weight structure $w$ on $\rigmot(K)$ satisfying the following conditions:
  \begin{enumerate}
  \item it is compatible with monoidal structure;
    
  \item it restricts to a bounded weight structure on $\rigmot(K)_{\omega}$ which is compatible with monoidal structure;
    
  \item $\rigmot(K)_{w \ge 0}$ is closed under small colimits;
    
  \item the natural inclusion $\rigmot_{\mathrm{gr}}(K) \subseteq \rigmot(K)$ is weight-exact, where $\rigmot_{\mathrm{gr}}(K)$ is equipped with the weight structure defined in \cite[Proposition 4.25]{BGV25}. More generally, for any finite Galois extension $L/K$, the fully faithful functor
    \[
\rigmot_{\mathrm{gr}}(L)^{\htcat \gal (L/K)} \to \rigmot(K)
\]
is weight-exact, where the source is equipped with weight structure $w_{L/K}$ in Lemma \ref{lem:wt-str-galinv}.
\end{enumerate}
\end{thm}
\begin{proof}
  For each finite Galois extension $L/K$, we have proved that $\rigmot_{L \text{-}\mathrm{gr}}(K)^{\htcat \gal (L/K)}$ is compactly generated by $\Nm_L \xi_L (M(X))$, where $X$ is proper smooth over $k_{L}$ (Proposition \ref{pro:cpt-gen-gal-inv-gd}). Under the identification $\prlc \simeq (\prrc)^{\opp}$, we deduce that
  \[
\colim \rigmot_{\mathrm{gr}}(L)^{\htcat \gal (L/K)}
\]
is compactly generated by images of $\Nm_L \xi_L (M(X))$ in this colimit. We let $\mathcal{N}$ be the collection of these objects. In order to use Corollary \ref{cor:wtstr-cg-from-negative}, it suffices to prove $\mathcal{N}$ is negative. This has been proved in Lemma \ref{lem:wt-str-galinv}. In particular, this also shows the fully faithful functors in ($4$) are weight-exact. Finally, this weight structure is compatible with the monoidal structure by (\ref{eq:wtmonoidal}).
\end{proof}

\begin{cor}
  \label{cor:pos-neg-rigwt}
  Let $w$ be the weight structure in Theorem \ref{thm:wt-str-rigmotcat}. Let $M \in \rigmot(K)$.
  \begin{enumerate}
  \item $M \in \rigmot(K)_{w \ge 0}$ if and only if, for any finite Galois extension $L/K$ and any $X/k_L$ proper smooth, we have $\hom(e_{*}\xi_L(M(X)), M[i])\simeq 0$ for all $i>0$, where $e$ is the structure map.
  \item $M \in \rigmot(K)_{\omega,w \le 0}$ if and only if, for any finite Galois extension $L/K$ and any $X/k_L$ proper smooth, we have $\hom(M[i], e_{*}\xi_L(M(X)))\simeq 0$ for all $i<0$, where $e$ is the structure map.
    
  \item The heart of $w$ on $\rigmot(K)$ (resp. $\rigmot(K)_{\omega}$) is the idempotent completion of full subcategory generated by all possible $e_{*} \xi_L(M(X))$ as above under direct sums (resp. finite direct sums).
\end{enumerate}
\end{cor}
\begin{proof}
These are immediate results of Corollary \ref{cor:wtstr-cg-from-negative}.
\end{proof}

\begin{cor}
  \label{cor:base-change-wtex}
  Let $L/K$ be an extension of complete non-archimedean fields with perfect residue fields $k_L/k$. Then the pullback functor
  \[
 e^{*}\colon \rigmot(K) \to \rigmot(L)
\]
is weight-exact, where $e\colon \spa(L) \to \spa(K)$ is the structure morphism. Furthermore, if $L/K$ is a finite extension, then
\[
e_{*}\colon \rigmot(L) \to \rigmot(K)
\]
is also weight-exact.
\end{cor}
\begin{proof}
  We firstly prove $e^{*}$ is weight-exact by using Corollary \ref{cor:cg-wtex}. We let $F/K$ be a finite Galois extension with residue field $k_{F}$ and $X$ be a proper smooth variety over $k_{F}$. Let $\alpha \colon \spa(F) \to \spa(K)$ be the structure map. We need to prove $e^{*} \alpha_{*} \xi_FM(X)$ lies in the heart of the weight structure of $\rigmot(L)$:
  \[
e^{*}\alpha_{*} \xi_F (M(X)) \simeq \xi_L \bar{e}^{*} \bar{\alpha}_{\sharp} M(X) \simeq \xi_{L}M(X \times_k k_{L}) \in \rigmot(L)^{\heartsuit_w}.
\]
  
For the weight-exactness of $e_{*}$ when $L/K$ is finite, we firstly assume $L/K$ is separable. As in Lemma \ref{lem:left-adjoint-G-inv}, we can know that $e_{*}$ is both the left and right adjoint of $e^{*}$. Therefore, we can conclude the weight-exactness of $e_{*}$ from the weight-exactness of $e^{*}$ thanks to \cite[Proposition 1.2.3 (9)]{Bon14}. In the general case, we can choose the separable closure $L_0$ of $K$ in $L$; hence $L/L_0$ is purely inseparable. Thus, we conclude the general case from \cite[Corollary 2.9.11]{AGV22} and the separable case.
\end{proof}


\section{Application: the Weight Spectral Sequence}
\label{sec:appl-weight-filtr}

In the section, we apply Theorem \ref{thm:wt-str-rigmotcat} to give the weight filtrations of Hyodo-Kato cohomologies $H^i_{\mathrm{HK}}(X)$ of smooth quasi-compact rigid analytic spaces $X$ over $K$.

In \S \ref{subsec:motivic-hydo-kato}, we review the motivic construction of Hyodo-Kato cohomologies developed in \cite{BGV25}, which take values in the derived $\infty$-category of $(\varphi,N, G_K)$-modules introduced in \S \ref{subsec:der-coeff}. In \S \ref{subsec:weight-filtrations}, we show that the motivic Hyodo-Kato realization factors through the weight complex functor, giving rise to a spectral sequence converging to Hyodo-Kato cohomologies. In particular, this induces the weight filtration on the Hyodo-Kato cohomology.

\begin{notation}
  \label{not:wtfil-setups}
  Throughout this section, we assume $K$ is a complete non-archimedean field with perfect residue field $k$. We fix an algebraic closure $\bar{K}$ of $K$ with the Galois group $G_K=\gal (\bar{K}/K)$. Let $C$ be the completion of $\bar{K}$ whose residue field is denoted by $\bar{k}$. We denote by $K_0= \frc W(k)$ the Arithmetic Frobenius $\sigma_{K}$ and write $\breve{K}:=\frc W(\bar{k})$. Following Theorem \ref{thm:wt-str-rigmotcat}, we denote by $w$ the weight structure on $\rigmot(K)$ whose heart is simply denoted by $\mathcal{H}_{K}$. In particular, there is a symmetric monoidal functor
\[
\subss{W}\colon \rigmot(K)_{\omega} \to \bkch(\htcat \mathcal{H}_{K}),
\]
called the weight complex functor; see Construction \ref{cons:wt-cplx}.
\end{notation}

\subsection{\texorpdfstring{The Derived $\infty$-Category of $(\varphi,N, G_K)$-Modules}{The Derived Infinite-Category of (Phi,N, GK)-Modules}}
\label{subsec:der-coeff}
We recall coefficients of cohomologies from \cite[\S 1.15]{Bei13}, \cite[\S 2.1]{DN18}, \cite{Fon94}, \cite[\S 8.2]{FO22}:

\begin{df}
  \label{df:phi-N-G-mod}
  Let $L/K$ be a finite Galois extension inside $\bar{K}$ with the Galois group $G_{L/K}$ and let $k_{L}$ be the residue field of $L$. We write $L_0=\frc W(k_L)$ and $\sigma_L\colon L_0 \to L_0$ for the arithmetic Frobenius of $L_0$.
  \begin{enumerate}
\item The category $\module{L_0}(\varphi,N, G_{L/K})$ of \myemph{$\boldsymbol{(\varphi,N, G_{L/K})}$-modules} over $L_0$ is the following category:
  \begin{itemize}
  \item objects are finite dimensional $L_0$-vector spaces $D$ equipped with
    {%
    \begin{itemize}
   \item a semi-linear $G_{L/K}$-action on $D$;
    \item a $\sigma_L$-semi-linear and $G_{L/K}$-equivariant bijection $\varphi_D\colon D \to D$;
    \item a $L_0$-linear and $G_{L/K}$-equivariant map $N_D \colon D \to D$ satisfying $N_D \varphi_D=p \varphi_D N_{D}$;
    \end{itemize}
    }%
  \item morphisms are $L_0$-linear maps commuting with $\varphi$, $N$ and $G_{L/K}$-actions.
  \end{itemize}
  \item If $L=K$, then the category of $(\varphi,N, G_{L/K})$-modules is called the category of \myemph{$\boldsymbol{(\varphi,N)}$-modules}, simply denoted by $\module{K_0}(\varphi,N)$. Moreover, in the case, the full subcategory spanned by those with $N=0$ is called the category of \myemph{$\boldsymbol{\varphi}$-modules} over $K_0$, and we will denote it by $\module{K_0}(\varphi)$.
\end{enumerate}
\end{df}

\begin{eg}
  \label{eg:isocry-twist}
\begin{enumerate}
\item For every $n \in \Z$, we define the \myemph{Tate twist} $K_0(n) \in \module{K_0}(\varphi)$ by the one-dimensional $K_0$-vector space $K_0$ with the twisted Frobenius $\varphi_{K_{0}(n)}=p^{-n} \sigma_{K}$.
  
\item Let $(D, \varphi_{D}) \in \module{K_0}(\varphi)$. Then, for every $n \in \Z$, we let $D(n)$ denote the tensor product $(D, \varphi_{D}) \otimes K_0(n)$ in $\module{K_0}(\varphi)$.
  
\item With the notation in ($2$), we can identify $(\varphi,N)$-modules over $K_0$ with $\varphi$-modules over $K_0$ with a morphism $N\colon D \to D(-1)$ in $\module{K_0}(\varphi)$.
\end{enumerate}
\end{eg}

\begin{notation}
  \label{not:dcoeff}
  Keeping notations as in Definition \ref{df:phi-N-G-mod}, we will denote the bounded derived $\infty$-category of $(\varphi,N,G_{L/K})$-modules over $L_0$ by $\bdcat_{(\varphi,N,G_{L/K})}(L_{0})$ and its Ind-completion by $\mathcal{D}_{(\varphi,N,G_{L/K})}(L_0)$. In particular, we have some special cases:
  \begin{enumerate}
  \item the derived $\infty$-category $\dcat_{(\varphi,N)}(K_{0})$ of $(\varphi,N)$-modules over $K_0$ when $L= K$;
  \item and similarly, we have the derived $\infty$-category $\dcat_{\varphi}(K_0)$ of $\varphi$-modules over $K_0$.
  \end{enumerate}
  If $L$ runs through all possible finite Galois extension of $K$, then we get
  \[
\bdcat_{(\varphi,N,G_K)}(\breve{K}):= \colim \bdcat_{(\varphi,N,G_{L/K})}(L_{0}) 
\]
where transition functors are induced by base change, and the Ind-completion of it is denoted by $\dcat_{(\varphi,N,G_K)}(\breve{K})$. We will call it the derived $\infty$-category of \myemph{$\boldsymbol{(\varphi,N,G_K)}$-modules} over $K$, which is in the sense of Fargues-Fontaine; see \cite[Définition 10.6.6]{thecurve} and the remark below.

These $\infty$-categories are indeed stable symmetric monoidal $\infty$-categories (see \cite[Proposition 3.2]{Dre15}, \cite[Corollary 4.8.1.14]{HA}).
\end{notation}

\begin{rmk}
  \label{rmk:phi-N-gal-mod}
In Definition \ref{df:phi-N-G-mod}, we did not use the assumption that $L/K$ is finite. Thus, there is a notion of $(\varphi,N,G_{K})$-modules defined as in Definition \ref{df:phi-N-G-mod}. Let us call them \myemph{naive $\boldsymbol{(\varphi,N,G_K)}$-modules}. However, the $\infty$-category $\dcat_{(\varphi,N,G_K)}(\breve{K})$ introduced above is not the derived $\infty$-category of all such naive $(\varphi,N,G_{K})$-modules. Rather, as mentioned, objects of $\bdcat_{(\varphi,N,G_K)}(\breve{K})$ defined above are discrete in the sense of Fontaine; that is, they form a special subclass of naive $(\varphi,N,G_K)$-modules; see \cite[$\mathrm{n}^{\circ}$ 4.2.2]{Fon94}. More precisely, a $(\varphi,N,G_K)$-module here is a naive $(\varphi,N,G_{K})$-module for which the inertia group $I_K$ acts with open stabilizers. This restriction does not affect the construction of the realization functors, since these functors always factor through the derived category $\dcat_{(\varphi,N,G_K)}(\breve{K})$.
\end{rmk}

\begin{rmk}
  \label{rmk:monadic-coeff}
  There is a categorical framework to construction the $\infty$-category $\dcat_{(\varphi,N)}(K_{0})$ from the $\infty$-category of $\varphi$-modules over $K_0$. To be precise, there is a canonical monoidal equivalence (\cite[Proposition 2.50]{BGV25})
  \begin{equation}
    \label{eq:phi-N-mod}
    \mathcal{D}_{(\varphi,N)}(K_0) \simeq \mathcal{D}_{\varphi}(K_0)_{\mathrm{nil}}^{-\otimes K_0(-1)}
  \end{equation}
  in $\calg(\prstc)$; where the right hand-side is defined in \cite[\S 2]{BGV25} whose objects are $\varphi$-modules $D$ together with a morphism $D \to D(-1)$, refereed as the \myemph{monodromy operator}, of $\varphi$-modules.
\end{rmk}

We next study some variants of Notation \ref{not:dcoeff}, which will be useful later.

\begin{notation}
  Let $L/K$ be a finite Galois extension. We define
  \[
\dcat_{(\varphi,G_{L/K})}(L_{0}):= \dcat_{\varphi}(L_0)^{\htcat G_{L/K}}
\]
where $\dcat_{\varphi}(L_{0}) \in \calg(\prstc)$ is acted on by the Galois group $G_{L/K}$ via the surjection $G_{L/K} \twoheadrightarrow G_{L_0/K_{0}}$. As above, let $L$ run through all possible finite Galois extensions; we can define
\[
\dcat_{(\varphi,G_K)}(\breve{K}):= \colim \dcat_{(\varphi,G_{L/K})}(L_{0})
\]
in $\calg(\prstc)$.
\end{notation}

For a finite Galois extension $L/K$, let $(D,\varphi_D,N_D,\rho)$ be a $(\varphi,N,G_{L/K})$-module over $L_0$. For each $g \in G_{L/K}$, with image $ \bar{g} \in G_{L_0/K_{0}}$, the map $\rho_g$ induced an isomorphism
\[
\rho_g \colon D \to \sigma_g^{*} D \simeq \bar{\sigma}_{g}^{*} D
\]
of $(\varphi,N)$-modules over $L_0$. Since forgetting the $G_{L/K}$-action is exact, we have a natural functor
\[
F \colon \dcat_{(\varphi,N,G_{L/K})}(L_0) \to \dcat_{(\varphi,N)}(L_0).
\]
For each $g \in G_{L/K}$, we also denote by $\sigma_g^{*}$ the corresponding autofunctors on $\dcat_{(\varphi,N)}(L_0)$. Then we have natural isomorphisms $F \xrightarrow{\simeq} \sigma_g^{*} \circ F$ for all $\sigma_g \in G_{L/K}$, which are compatible with the group multiplication in $G_{L/K}$ (since the same holds for $\rho$). In particular, this yields a functor
\[
\dcat_{(\varphi,N,G_{L/K})}(L_0) \to \dcat_{(\varphi,N)}(L_0)^{\htcat G_{L/K}}
\]
in $\calg(\prstc)$.

\begin{pro}
  \label{pro:G-inv-phi-N}
  Let $L/K$ be a finite Galois extension. The natural functor
  \[
\dcat_{(\varphi,N,G_{L/K})}(L_0) \xrightarrow{\simeq} \dcat_{(\varphi,N)}(L_0)^{\htcat G_{L/K}}
\]
defined above is an equivalence in $\calg(\prstc)$. In particular, we get the functor forgetting monodromy operators
\[
  \dcat_{(\varphi,N,G_{L/K})}(L_0)\xrightarrow{\simeq} \dcat_{(\varphi,N)}(L_0)^{\htcat G_{L/K}} \to \dcat_{\varphi}(L_0)^{\htcat G_{L/K}} = \dcat_{(\varphi,G_{L/K})}(L_{0}).
\]
\end{pro}
\begin{proof}
  It suffices to prove by restricting to the bounded parts. We first show that it is fully faithful. Let $D$ and $E$ be bounded complexes of $(\varphi,N,G_{L/K})$-modules. Then, by \cite[(2.3)]{DN18}, the mapping space $\map_{\bdcat_{(\varphi,N,G_{L/K})}(L_{0})}(D,E)$ is the homotopy limit of the diagram
  \[
\begin{tikzcd}
{\map_{\bdcat(L_0)}(D,E)^{\htcat G_{L/K}}} \arrow[r, "{\varphi_{E,*}-\varphi_D^*}"] \arrow[d, "{N_{E,*}-N_D^*}"'] & {\map_{\bdcat(L_0)}(D,\sigma_{L,*}E)^{\htcat G_{L/K}}} \arrow[d, "{N_{E,*}-pN_D^*}"] \\
{\map_{\bdcat(L_0)}(D,E)^{\htcat G_{L/K}}} \arrow[r, "{p\varphi_{E,*}-\varphi_D^*}"']                             & {\map_{\bdcat(L_0)}(D,\sigma_{L,*}E)^{\htcat G_{L/K}}}                              
\end{tikzcd}
\]
Since limits commute with limits, we know that
\[
\map_{\bdcat_{(\varphi,N,G_{L/K})}(L_0)}(D,E) \simeq \map_{\bdcat_{(\varphi,N)}(L_0)}(D,E)^{\htcat \gal (L/K)}
\]
by \cite[Lemma 2.5]{DN18} or \cite[\S 1.15]{Bei13}. The full faithfulness follows from Corollary \ref{cor:mapsp-htfixpt}. The essential surjectivity is deduced from the construction of the functor and the definition of $\dcat_{(\varphi,N)}(L_0)^{\htcat G_{L/K}}$, whose objects are equivalently given by an object in $\dcat_{(\varphi,N)}(L_0)$ equipped with an action of $G_{L/K}$ satisfying the higher coherence.

Finally, forgetting monodromy operators is compatible with the action of $G_{L/K}$, hence forgetful functors are well-defined.
\end{proof}

\subsection{Motivic Hyodo-Kato Cohomology}
\label{subsec:motivic-hydo-kato}
There is a motivic analogy of Remark \ref{rmk:monadic-coeff}: we have a monoidal equivalence
    \[
\agmot(k)^{-\otimes {\mathbbm 1}(-1)}_{\mathrm{nil}} \simeq \module{{\mathbbm 1} \oplus {\mathbbm 1}(-1)} (\agmot(k)),
\]
see \cite[Proposition 2.32, Corollary 4.14]{BGV25}. If we choose a pseudo-uniformizer $\varpi$, we can identify the latter with $\rigmot_{\mathrm{gr}}(K)$ by Proposition~\ref{pro:rigmot-as-module} and \cite[Corollary 3.8.32]{AGV22}.

From this perspective, there is a motivic Hyodo-Kato cohomology\footnote{The effect of different choices of pseudo-uniformizer $\varpi$ has been explained in \cite[Remark 4.18]{BGV25}, and we also refer readers to \cite[Remark 3.4]{BKV25} below for details in the log approach.}:
\[
\rgama{\mathrm{HK}}^{\varpi}\colon \rigmot_{\mathrm{gr}}(K) \simeq \agmot(k)_{\mathrm{nil}}^{-\otimes {\mathbbm 1}(-1)} \to \mathcal{D}_{\varphi}(K_0)_{\mathrm{nil}}^{-\otimes K_0(-1)} \simeq \mathcal{D}_{(\varphi,N)}(K_{0})
\]
which is induced by adding the monodromy operators in the motivic rigid cohomology and takes values in the derived $\infty$-category of $\varphi$-modules. For further details, we refer to \cite[\S 4]{BGV25}, especially \cite[Lemma 4.30]{BGV25} for the middle functor above. We will refer to this functor as the \myemph{$\varpi$-Hyodo-Kato realization}.

\begin{rmk}
  \label{rmk:covariant-HK}
In general, cohomology theories are contravariant functors. However, the motivic Hyodo–Kato cohomology considered here is a covariant functor. This is because it is defined by taking the Hyodo–Kato cohomology of the duals of motives. Accordingly, the realization functor used in this context computes the dual of the usual Hyodo-Kato cohomology; see Proposition~\ref{pro:HK-coh} below.
\end{rmk}

\begin{pro}
  \label{pro:HK-coh}
  Let $\varpi $ be a pseudo-uniformizer of $\mathcal{O}_{K}$.
  \begin{enumerate}
  \item The functor $\pi\circ \rgama{\mathrm{HK}}^{\varpi} \circ \xi \colon \agmot(k) \to \dcat_{\varphi}(K_0)$ computes the rigid cohomology, where $\pi$ is the functor forgetting the monodromy. In particular, if $k$ is finite, for any proper smooth variety $X$ over $k$, the $i$-cohomology (or $(-i)$-homology) of $\pi \circ \rgama{\mathrm{HK}}^{\varpi} \circ \xi(M(X)^{\vee})$ is pure of weight $i$ (see \cite[Definition 2.42]{BGV25}).

  \item Assume $k$ is a finite field. Then the functor $\rgama{\mathrm{HK}}^{\varpi}$ computes the arithmetic overconvergent Hyodo-Kato cohomology in the sense of \cite[\S 5.2.2]{CN20} if $\Char K=p>0$ or $\Char K=0$ with $\varpi=p$.
\end{enumerate}
\end{pro}
\begin{proof}
The part ($1$) is \cite[Proposition 4.44]{BGV25}, and see also \cite{KM74} and Remark \ref{rmk:covariant-HK}. The part ($2$) follows from \cite[Corollary 4.58]{BGV25} and \cite[Proposition 3.30]{BKV25}.
\end{proof}

\begin{notation}
  In light of Proposition \ref{pro:HK-coh} ($2$), we will adopt the notation
\[
\rgama{\mathrm{HK},K}\colon \rigmot_{\mathrm{gr}}(K) \to \dcat_{(\varphi,N)}(K_{0})
\]
for $\rgama{\mathrm{HK}}^{\varpi}$ when $\Char K=p>0$ or $\Char K=0$ with $\varpi=p$. In other words, when we use this abbreviation, we have already chosen a specific pseudo-uniformizer.
\end{notation}

\begin{cons}
  \label{cons:arith-HK}
  We next recall how to enhance the Hyodo-Kato realization to take value of $(\varphi,N,G_{K})$-modules; see also \cite[Corollary 3.35]{BKV25}: we have a realization, denoted by $\rgama{\mathrm{HK}}^{\mathrm{ari}}$
    \begin{align*}
    \rigmot(K)\simeq \colim \rigmot_{\mathrm{gr}}(L)^{\htcat G_{L/K}} \to \colim \dcat_{(\varphi,N)}(L_0)^{\htcat G_{L/K}} &\simeq \colim \dcat_{(\varphi,N,G_{L/K})}(L_{0})\\
     &\simeq \dcat_{(\varphi,N,G_K)}(\breve{K}),
\end{align*}
where the first equivalence is Proposition \ref{pro:rigmot-cptgen-pgd} and the second equivalence is Proposition \ref{pro:G-inv-phi-N}. Furthermore, we can forget monodromy operators:
\begin{equation}
  \label{eq:phi-GK-HK}
  \rgama{\mathrm{HK}}^{(\varphi,G_{K})} \colon \rigmot(K) \to \colim \dcat_{(\varphi,N)}(L_0)^{\htcat G_{L/K}} \to \colim \dcat_{(\varphi,G_{L/K})}(L_0) \simeq \dcat_{(\varphi,G_K)}(\breve{K}).
\end{equation}
\end{cons}


\subsection{The Weight Filtration on the Hyodo-Kato Cohomology}
\label{subsec:weight-filtrations}
In the final part, we use the weight complex functor (see Construction \ref{cons:wt-cplx}) to construct a spectral sequence and deduce the weight filtration on the Hyodo-Kato realizations of compact rigid analytic motives over $K$. This can be regarded as the analogue\footnote{The semi-stable reduction case is already given in \cite{BGV25}.} of the weight spectral sequence studied in \cite{RZ82} for rigid analytic spaces. But we do not assume the existence of any formal models.

From now on, we assume the residue field $k$ of $K$ is a finite field.

\begin{lem}
  \label{lem:GHK-wtcplx}
The realization functor (\ref{eq:phi-GK-HK})
  \[
\rgama{\mathrm{HK}}^{(\varphi,G_{K})} \colon \rigmot(K)_{\omega} \to \bdcat_{(\varphi,G_K)}(\breve{K})
\]
factors through the weight complex functor $\subss{W}\colon \rigmot(K)_{\omega} \to \bkch(\htcat \mathcal{H}_{K})$.
\end{lem}
\begin{proof}
  In the proof, we will simply write $\rgama{\mathrm{HK}}$ for $\rgama{\mathrm{HK}}^{(\varphi,G_{K})}$ and use $\map_{(\varphi,G_K)}$ or $\map_{\varphi}$ to denote the mapping spaces between $(\varphi,G_K)$-modules or $\varphi$-modules respectively. By \cite[Corollary 3.30]{BGV25} and Corollary \ref{cor:pos-neg-rigwt} ($3$), it suffices to check that, for any finite Galois extensions $L/K$ and $L'/K$ and any proper smooth varieties $X$ over $k_L$ and $Y$ over $k_{L'}$, we have
  \begin{equation}
    \label{eq:factor-wtcplx}
    \pi_i \map_{(\varphi,G_K)}(\rgama{\mathrm{HK}}(e_{*}\xi_LM(X)), e_{*}' \xi_{L'}M(Y)) \simeq 0
  \end{equation}
for $i>0$. The construction (\ref{eq:phi-GK-HK}) shows that $\rgama{\mathrm{HK}}(e_{*} \xi_LM(X))$ is the image of $\rgama{\mathrm{HK},L}(e^{*}e_{*}\xi_LM(X))$ under the canonical functor
\[
\dcat_{(\varphi,G_{L/K})}(L_0) \to \dcat_{(\varphi,G_K)}(\breve{K}).
\]
We first compute in $\bdcat_{(\varphi,G_{L/K})}(L_{0})$:
\begin{equation}
  \label{eq:HK-K-heart}
 \rgama{\mathrm{HK},L}(e^{*}e_{*} \xi_L M(X)) \simeq \bigoplus_{e_{L/K}}   \rgama{\mathrm{HK},L}(\xi_LM(X))
\end{equation}
where the equivalence follows from Proposition \ref{pro:proper-sm-under-lad}. Now let $F/K$ be a finite Galois extension containing both $L$ and $L'$. Note that the canonical functor $\dcat_{(\varphi,G_{L/K})}(L_0) \to \dcat_{(\varphi,G_K)}(\breve{K})$ factors through $\dcat_{(\varphi,G_{F/K})}(F_0)$. This, together with (\ref{eq:HK-K-heart}), yields that $\rgama{\mathrm{HK}}(e_{*}\xi_LM(X))$ is the image of
\[
\bigoplus_{e_{L/K}} \rgama{\mathrm{HK},F}(\xi_FM(X_{k_{F}}))
\]
in $\dcat_{(\varphi,G_K)}(\breve{K})$; and the similar fact holds for $M(Y)$. Therefore, we have in $\dcat_{(\varphi,G_{F/K})}(F_0)$
\begin{align*}
  \phantom{\simeq\,}&\pi_i \map_{(\varphi,G_{F/K})}(\rgama{\mathrm{HK},F}(\xi_FM(X_{k_{F}})), \rgama{\mathrm{HK},F}(\xi_FM(Y_{k_{F}})))\\
  \simeq\,
&\pi_i \map_{\varphi}(\rgama{\mathrm{HK},F}(\xi_FM(X_{k_{F}})), \rgama{\mathrm{HK},F}(\xi_FM(Y_{k_{F}})))^{\htcat G_{F/K}}\simeq 0
\end{align*}
for $i>0$, where the vanishing follows from \cite[Example 3.31]{BGV25} together with Proposition \ref{pro:HK-coh}. Since this holds for arbitrary such finite Galois extension $F$ over $K$, this proves (\ref{eq:factor-wtcplx}) holds.
\end{proof}

\begin{pro}
  \label{pro:wtfil}
   Let $M$ be a compact motive in $\rigmot(K)$. There is a convergent spectral sequence starting at the first page and degenerating at the second page:
  \[
E_{pq}^1(M)=H_q \rgama{\mathrm{HK}}^{(\varphi,G_K)}(W_pM) \Rightarrow H_{p+q} \rgama{\mathrm{HK}}^{(\varphi,G_K)}(M),
\]
where the differential maps are induced by differential maps in the weight complex. In particular, we have a (finite) increasing filtration on $H_n \rgama{\mathrm{HK}}^{(\varphi,G_K)}(M)$ whose $i$-th graded piece is pure of weight $i-n$. Moreover, the filtration is stable under $G_K$-action, and the monodromy induces a map
\[
\mathrm{gr}^W_i H_{n} \rgama{\mathrm{HK}}^{(\varphi,G_K)}(M) \to \mathrm{gr}^W_{i-2} H_n \rgama{\mathrm{HK}}^{(\varphi,G_K)}(M)(-1).
\]
\end{pro}
\begin{proof}
  The naive truncation gives a filtration
  \[
\cdots \to \tau_{\le i-1} \subss{W}M \to \tau_{\le i} \subss{W}M \to \tau_{\le i+1} \subss{W}M
  \]
  for $\subss{W}M$. Applying the functor $F$, we get a filtration on $\rgama{\mathrm{HK}}^{(\varphi,G_K)}(M)$. Then, using the spectral sequence \cite[Proposition 1.2.2.14]{HA}, we get the desired spectral sequence. For the degeneracy, it suffices to look at the underlying morphisms between $\varphi$-modules (using Corollary \ref{cor:mapsp-htfixpt}). As $\varphi$-modules, each term $E_{pq}^1$ is pure of weight $-q$ by Proposition \ref{pro:HK-coh} and (\ref{eq:HK-K-heart}). Thus, on the second page, each differential map $E_{pq}^2 \to E_{p-2,q+1}^{2}$ is a morphism of $\varphi$-modules with different weights, and hence it vanishes. As a consequence, the $\infty$-term $E_{pq}^{\infty}$ is pure of weight $-q$, which is the $p$-th graded of the induced filtration on $H_{p+q} \rgama{\mathrm{HK}}^{(\varphi,G_K)}(M)$.

  For last assertion, the $G_K$-action is clear. We need to prove for the monodromy operator. As $M$ is compact, Proposition \ref{pro:rigmot-cptgen-pgd} shows that the motive $M$ lies in a full subcategory $\rigmot_{\mathrm{gr}}(L)^{\htcat G_{L/K}}$ for some finite Galois extension $L$ of $K$. As we proved, the fully faithful functor
  \[
\rigmot_{\mathrm{gr}}(L)^{\htcat G_{L/K}}_{\omega} \hookrightarrow \rigmot(K)_{\omega}
\]
is weight-exact (Theorem \ref{thm:wt-str-rigmotcat} ($4$)). Therefore, the weight complex $\subss{W} M$ of $M$ comes from the weight complex of $M$ in $\rigmot_{\mathrm{gr}}(L)^{\htcat G_{L/K}}$. The proof of Lemma \ref{lem:GHK-wtcplx} also proves that the realization
\[
 \rgama{\mathrm{HK},L}^{(\varphi,G_{L/K})} \colon \rigmot_{\mathrm{gr}}(L)^{\htcat G_{L/K}}_{\omega} \to \bdcat_{(\varphi,G_{L/K})}(L_{0})
\]
factors through the weight complex of $w_{L/K}$ defined in Lemma \ref{lem:wt-str-galinv}. Therefore, the spectral sequence associated to $M$ is the image of the spectral sequence associated to $M$ in $\rigmot_{\mathrm{gr}}(L)^{\htcat G_{L/K}}$. It suffices to restrict to this full subcategory and prove the monodromy induces the desired map between graded pieces. The realization
\[
\rgama{\mathrm{HK},L} \colon \rigmot_{\mathrm{gr}}(L)^{\htcat G_{L/K}} \to \dcat_{(\varphi,N,G_{L/K})}(L_{0})
\]
is obtained by applying the $G_{L/K}$-homotopy fixed points functor to the Hyodo-Kato realization on $\rigmot_{\mathrm{gr}}(L)$. Thus, we have a commutative diagram
\[
\begin{tikzcd}
\rigmot_{\mathrm{gr}}(L)^{\htcat G_{L/K}}_{\omega} \arrow[d, "(-)_L"'] \arrow[r, "\subss{W}^L"] & \bkch(\htcat \mathcal{H}_{L/K}) \arrow[d] \arrow[r]   & {\bdcat_{(\varphi,G_{L/K})}(L_0)} \arrow[d] & {\bdcat_{(\varphi,N,G_{L/K})}(L_0)} \arrow[d] \arrow[l] \\
\rigmot_{\mathrm{gr}}(L)_{\omega} \arrow[r, "\subss{w}^L"]                                      & \bkch(\htcat \mathcal{H}_L^{\mathrm{chow}}) \arrow[r] & \bdcat_{\varphi}(L_0)                       & {\bdcat_{(\varphi,N)}(L_0)} \arrow[l]                  
\end{tikzcd}
\]
where $\subss{W}^{L}$ is the weight complex functor with respect to $w_{L/K}$ and $\subss{w}^{L}$ is the weight complex functor with respect to the weight structure on $\rigmot_{\mathrm{gr}}(L)_{\omega}$. Note that the weight spectral sequence of $M$ on $\rigmot_{\mathrm{gr}}(L)^{\htcat G_{L/K}}$ is given by
\[
E_{pq}^1= H_q \rgama{\mathrm{HK},L}^{(\varphi,G_{L/K})}(W^L_pM) \Rightarrow H_{p+q} \rgama{\mathrm{HK},L}^{(\varphi,G_{L/K})}(M).
\]
We need to show the monodromy operator on $H_{n} \rgama{\mathrm{HK},L}(M)$ induces a map
\[
\mathrm{gr}_i H_n \rgama{\mathrm{HK},L}^{(\varphi,G_{L/K})}(M) \to \mathrm{gr}_{i+2} H_n \rgama{\mathrm{HK},L}^{(\varphi,G_{L/K})}(M)(-1).
\]
For this, it suffices to look at their underlying $\varphi$-modules. In other words, the spectral sequence becomes
\[
E_{pq}^1=H_q \rgama{\mathrm{HK},L}(w^L_p (M_L)) \Rightarrow H_{p+q} \rgama{\mathrm{HK},L}(M_L)
\]
which is exactly the weight spectral sequence associated to $M_L$ studied in \cite[Example 4.35, 4.36]{BGV25}, where it is already shown that the monodromy operator induces a map from the $i$-th graded piece to the $(i+2)$-th graded piece.
\end{proof}

\begin{rmk}
  The spectral sequence in Proposition~\ref{pro:wtfil} is formulated homologically, but it can be translated into a cohomological version. Throughout this paper we adopt the homological convention, as it is more natural and convenient in the context of stable $\infty$-categories, where the pervasive use of homotopy-theoretic structures makes the homological perspective more suitable.
\end{rmk}

If we take $M$ to be a motive associated to a smooth quasi-compact $K$-rigid analytic space, then we get the weight filtration on $H^i_{\mathrm{HK}}(X)$ by the fact that $H^i_{\mathrm{HK}}(X)= H_{-i} \rgama{\mathrm{HK}}^{(\varphi,G_K)} (M(X)^{\vee})$ (see Remark \ref{rmk:covariant-HK}).

\begin{cor}[Weight Filtrations for Rigid Analytic Spaces]
  \label{cor:wt-fil-rigsp}
 Let $X$ be a smooth quasi-compact $K$-rigid analytic space. Then, for each $i$, the (overconvergent) Hyodo-Kato cohomology $H^i_{\mathrm{HK}}(X)$ admits a finite increasing filtration $\mathrm{Fil}^W_k H^i_{\mathrm{HK}}(X)$ stable under $G_K$-action whose $k$-th graded piece is pure of weight $i+k$ and the monodromy induces a map $\mathrm{gr}_k^{W} H^i_{\mathrm{HK}}(X) \to \mathrm{gr}^W_{k-2} H^i_{\mathrm{HK}}(X)$.
\end{cor}

\begin{rmk}
  \label{rmk:G-wtfil}
  The filtration in Corollary \ref{cor:wt-fil-rigsp} is referred as the \myemph{weight filtration} of the Hyodo-Kato cohomology $H^i_{\mathrm{HK}}(X)$. In \cite{BGV25}, the authors constructed it for a special case where $X$ admits semi-stable reduction, which can be regarded as the $p$-adic analytic analogue of the weight filtration defined by Rapoport-Zink in \cite{RZ82} for $\ell$-adic cohomologies of algebraic varieties. In contrast, we do not assume the existence of formal models of $X$, and we prove that our weight filtration is stable under the Galois action.

  This filtration is closely related to Deligne's weight-monodromy conjecture. Indeed, it satisfies the conditions for Deligne's weight filtration as formulated in \cite{Del70,Del74} and partially satisfies those for the \myemph{monodromy filtration} in \cite{Del74}. More precisely, the only missing piece is whether the monodromy operator induces an isomorphism
  \begin{equation}
    \label{eq:monofil}
    N^k \colon \mathrm{gr}^W_k H^i_{\mathrm{HK}}(X) \to \mathrm{gr}^W_{-k} H^i_{\mathrm{HK}}(X).
  \end{equation}
From this perspective, if the morphism in \eqref{eq:HK-K-heart} is an isomorphism, then the $p$-adic weight-monodromy conjecture holds.
\end{rmk}


\nocite{Sau23,berkeley,kashiwara,Ito05,Ito05a,Rot09,Wei94,Sch12,RZ82,RS20,Sai03}

\titleformat{\section}[display]
{\normalfont\Large\bfseries\filcenter}
{\thesection}%
{1pc}
{\titlerule[2pt]
  \vspace{1pc}%
  \huge}

\printbibliography[heading=bibintoc]
\end{document}
